\documentclass[11pt]{amsart}
\usepackage{amssymb,amsmath,epsfig,color,comment}
\usepackage{hyperref}

\newtheorem{theorem}{Theorem}[section]
\newtheorem{lemma}[theorem]{Lemma}
\newtheorem{proposition}[theorem]{Proposition}
\newtheorem{corollary}[theorem]{Corollary}

\newtheorem{conjecture}[theorem]{Conjecture}
\theoremstyle{definition}
\newtheorem{definition}[theorem]{Definition}
\newtheorem{example}[theorem]{Example}
\newtheorem{remark}[theorem]{Remark}

\title{Khovanov homology, wedges of spheres and complexity}

\author{Jozef H. Przytycki}
\address{The George Washington University, US \& University of Gdańsk, Poland}
\email{przytyck@gwu.edu}
\author{Marithania Silvero}			
\address{Universidad de Sevilla, Spain}
\email{marithania@us.es}

\begin{document}
\maketitle

\begin{center}
\textbf{Abstract}
\end{center}
\vspace{0.2cm}
Our main result has topological, combinatorial and computational flavor. It is motivated by a fundamental conjecture stating that computing Khovanov homology of a closed braid of fixed number of strands has polynomial time complexity. We show that the independence simplicial complex $I(w)$  associated to the 4-braid diagram $w$ (and therefore its Khovanov spectrum at extreme quantum degree) is contractible or homotopy equivalent to either a sphere, or a wedge of 2 spheres (possibly of different dimensions), or a wedge of 3 spheres (at least two of them of the same dimension), or a wedge of 4 spheres (at least three of them of the same dimension). On the algorithmic side we prove that finding the homotopy type of $I(w)$ can be done in polynomial time with respect to the number of crossings in $w$. In particular, we prove the wedge of spheres conjecture for circle graphs obtained from 4-braid diagrams. We also introduce the concept of Khovanov adequate diagram  and discuss criteria for a link to have a Khovanov adequate braid diagram with at most 4 strands.

\tableofcontents


\section{Introduction} 
Our main result has topological, combinatorial and computational flavor and is connecting four fundamental conjectures:
\begin{conjecture}\label{Conjecture 1}\
\begin{enumerate}
\item [(a)] Computing Khovanov homology of a closed braid with fixed number of strands has polynomial time complexity with respect to the number of crossings.
\item[(b)] Determining the homotopy type of the geometric realization of  Khovanov homology (Khovanov spectra) of a closed braid with fixed number of strands has polynomial time complexity with respect to the number of crossings. 
\end{enumerate}
\end{conjecture}

A fundamental difference between Alexander polynomial and Jones (and HOMFLYPT) polynomial is that Alexander polynomial can be computed in polynomial time while finding Jones (and HOMFLYPT) polynomial is NP-hard \cite{Ja,Wel}. Thus computing Khovanov homology, a categorification of Jones polynomial, is NP-hard. At the moment, all existing programs computing Khovanov homology have exponential complexity (compare \cite{BN}). Therefore finding an algorithm of polynomial time complexity for Khovanov homology of closed braids with a fixed number of strands, would be the game changer allowing posing and testing conjectures about structure of Khovanov homology. 

Initial motivation for Conjecture \ref{Conjecture 1} (a) comes from the fact that computing HOMFLYPT (and therefore Jones) polynomial of a closed braid with fixed number of strands has polynomial time complexity. Such polynomial growth algorithm was developed by Morton and Short in \cite{MS}, even if complexity was not discussed in that paper.  

 The second conjecture concerns the geometric realization of extreme Khovanov homology. Lipshitz and Sarkar introduced a graded family of spectra associated to a link diagram refining Khovanov homology \cite{lipshitzsarkar}, called Khovanov spectra. Independently, in \cite{GMS} it was introduced a method to associate to every link diagram $D$ a simplicial complex $I_D$ so that its cohomology equals  Khovanov homology of $D$ at extreme quantum grading. This construction was proven to be equivalent to that of Lipshitz and Sarkar in extreme quantum grading \cite{FedSil}. The following conjecture was formulated in \cite{JACO}. 

\begin{conjecture}\label{Conjecture 2} The geometric realization of the extreme Khovanov homology (extreme Khovanov spectrum) is homotopy equivalent to a wedge of spheres (allowing empty wedge, i.e., a contractible set).
\end{conjecture}

The construction in \cite{GMS} involves the independence complex of certain circle graphs (not necessarily connected) called Lando graphs. By construction, Lando graphs turn out to be bipartite, so they are a proper subset of the family of circle graphs. The third conjecture, formulated also in \cite{JACO}, generalizes Conjecture \ref{Conjecture 2} and it is formulated purely in the language of algebraic combinatorics.

\begin{conjecture}\label{Conjecture 3} (Wedge of spheres conjecture) The independence complex of a circle graph is homotopy equivalent to a wedge of spheres.
\end{conjecture}

The fourth conjecture is essentially that of Adamaszek (compare to \cite{Adamacomplexity}). However, he formulated it as a question, but our calculations and partial results gave us support to state it as a conjecture. Note that, by Theorem~\ref{teoGMS}, Conjecture \ref{Conjecture 4} implies Conjecture \ref{Conjecture 1} (b) at extreme quantum grading.

\begin{conjecture}\label{Conjecture 4}
The homotopy type of the independence complex of a circle graph can be found in polynomial time with respect to the number of vertices of the graph.
\end{conjecture}

In this paper we solve Conjecture \ref{Conjecture 1} in the case of extreme Khovanov homology and its geometric realization for 4-strands braids. As a byproduct of the paper we also solve Conjectures \ref{Conjecture 2}, \ref{Conjecture 3} and \ref{Conjecture 4} in those cases related to 4-strands braids. Our main result is the following:

\begin{theorem}
Let $w$ be a braid diagram on $4$ strands and write $\hat{w}$ for its closure. Then, the homotopy type of the geometric realization $I(w)$ of the extreme Khovanov homology of $\hat{w}$ can be computed in polynomial time. Moreover, if not contractible, $I(w)$ is homotopy equivalent to either a sphere, or a wedge of two spheres, or $S^k \vee S^i \vee S^i$, or $S^k \vee S^i \vee S^i \vee S^i$, with $k \geq i$. 
\end{theorem}

As a consequence of the above result, we get the following corollary, that can be used to find criteria for links to have Khovanov adequate braid diagrams on 4 strands (see Section \ref{sectionkhovad}). 

\begin{corollary}\label{summarykhov}
Let $\beta$ be a 4-strands braid diagram and $\hat{\beta}$ its closure. Then its extreme Khovanov homology $Kh_{*,j_{\min}}(\hat{\beta})$ is trivial or equal to $\mathbb{Z}$, to $\mathbb{Z}[k] \oplus \mathbb{Z}[i]$, to $\mathbb{Z}[k]\oplus \mathbb{Z}^2[i]$, or to $\mathbb{Z}[k]\oplus \mathbb{Z}^3[i]$, where $k\geq i$ and $[h]$ denotes the homological grading of Khovanov homology. 
\end{corollary}

Observe that results in this paper are stated in terms of unoriented (framed) Khovanov homology (see Section 2 and compare to \cite{Vir1}). To get analogous result in terms of oriented version of Khovanov homology, one needs to adjust gradings as indicated in Remark \ref{remarkgradings}. \\ 

The paper is organized as follows:

In Section \ref{Section 2} we briefly review the definition of Khovanov homology of unoriented (framed) links following Viro \cite{Vir1} and recall the geometric realization of extreme Khovanov homology introduced in \cite{GMS}.  

Section 3 is devoted to independence complexes of graphs. After recalling basic properties, we analyze the homotopy type of complicated graphs (e.g. augumented rhomboid graphs), which will be crucial in the proof of Theorem~1.5. 

In Section \ref{Section 5} we introduce some basic definitions concerning the braid monoid $M_n$ and Temperley-Lieb monoid $M^{TL}_n$; this allows us to explain the scheme of the proof of Theorem 1.5 in Section \ref{Section 6}. The proof is completed in Sections \ref{secstronglyred}, \ref{secpositive} and \ref{sectnegative}.

We finish our paper with some applications and concluding remarks in Section \ref{secfinal}.


\section{Khovanov homology}\label{Section 2}

We briefly review the definition of Khovanov homology of unoriented (framed) links following Viro \cite{Vir1} and recall the geometric realization of extreme Khovanov homology introduced in \cite{GMS}. 

Let $D$ be an unoriented link diagram with $c$ ordered crossings. A \textit{Kauffman state} $s$ assigns a label, $A$ or $B$, to each crossing of $D$, that is, \, $s:~cr(D)~\to \{A,B\}$. Set $\sigma(s) = |s^{-1}(A)|- |s^{-1}(B)|$ and let $\mathcal{S}$ be the collection of $2^{c}$ possible states of $D$. The \textit{resolution} $sD$ is the system of circles and chords obtained after smoothing each crossing of $D$ according to the label assigned by $s$ by following the convention in Figure \ref{markers}. We write $|s|$ for the number of circles in $sD$.

\begin{figure}
\centering
\includegraphics[width = 9cm]{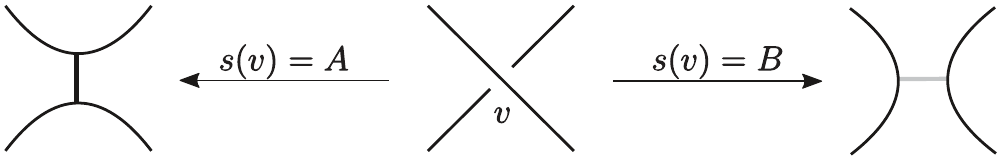}
\caption{\small{Smoothing of a crossing according to its $A$ or $B$-label. $A$-chords (resp. $B$-chords) are represented by dark (resp. light) segments.}}
\label{markers}
\end{figure}

An \textit{enhancement} of a state $s$ is a map $\epsilon$ assigning a sign $\pm 1$ to each of the circles in $sD$. We sometimes keep the letter $s$ for the enhanced state $(s, \epsilon)$ to avoid cumbersome notation. Write $\tau(s, \epsilon) =\Sigma \epsilon(r)$, where $r$ ranges over all circles in $sD$, and define, for the enhanced state $(s,\epsilon)$, the integers 
$$i(s, \epsilon) = i(s) = \sigma(s), \quad \quad j(s, \epsilon) = \sigma(s) + 2 \tau(s, \epsilon).$$

The enhanced state $s'$ is said to be \textit{adjacent} to the enhanced state $s$ if $i(s')=i(s)-2$, \, $j(s')=j(s)$, both states have identical labels except for one crossing $v$, where $s$ assigns an $A$-label and $s'$ a $B$-label, and they assign the same sign to the common circles in resolutions $sD$ and $s'D$.

Let $C_{i,j}(D)$ be the free abelian group generated by the set of enhanced states $s$ of $D$ with $i(s)=i$ and $j(s)=j$. For each integer $j$, consider the chain complex $$\ldots \, \longrightarrow \, C_{i,j} \, \stackrel{\partial_{i,j}}{\longrightarrow} \, C_{i-2,j} \,\longrightarrow \, \ldots $$

\noindent with differential $\partial_{i,j}(s) = \Sigma_{s'\in \mathcal{S}} (s:s') s'$, with $(s:s')=0$ if $s'$ is not adjacent to $s$ and otherwise $(s:s')=(-1)^k$, with $k$ the number of $B$-labeled crossings in $s$ coming after crossing $v$ in the chosen ordering. 

It turns out that $\partial_{i-2,j}\circ \partial_{i,j}=0$ and the corresponding homology groups
$$Kh_{i,j}(D) = \frac{ker(\partial_{i,j})}{im(\partial_{i+2,j})}$$ are invariants of framed unoriented links, and they categorify the unreduced Kauffman bracket polynomial. We refer to them as (framed) Khovanov homology groups of $D$.

\begin{remark}\label{remarkgradings}
The framed unoriented version of Khovanov homology is equivalent to its oriented version, which categorifies Jones polynomial. Framed and oriented version are related by $Kh_{i,j}(D) = H^{I,J}(\vec{D})$, where $I= \frac{w-i}{2}$ and $J=\frac{3w-j}{2}$, with $w$ the writhe of the oriented diagram $\vec{D}$.  
\end{remark}

Let $j_{\min}(D) = \min \{j(s,\epsilon) \, | \, (s,\epsilon) \mbox{ is an enhanced state of } D\}.$ We will refer to the complex $\{C_{i,j_{\min}}(D), \partial_i\}$ as the extreme Khovanov complex and to the corresponding homology groups $Kh_{i, j_{\min}}(D)$ as the (potential) extreme Khovanov homology groups of $D$. If we denote by $s_B$ the state assigning a $B$-label to every crossing of $D$, then $j_{\min}(D)=-c-2|s_B|$.

\begin{remark}
Note that the integer $j_{\min}(D)$ depends on the diagram and may differ for two different diagrams representing the same link. It may happen that $Kh_{*,j_{\min}}(D)=0$.
\end{remark}


Next, we review the construction from \cite{GMS} to describe a geometric realization of the extreme Khovanov complex in terms of certain simplicial complex.  

\begin{definition}\label{defLando}
Given a link diagram $D$, write $s_B$ for the state assigning a $B$-label to every crossing of $D$. The Lando graph of $D$, $G_D$, is constructed from $s_BD$ by associating a vertex to each chord having both endpoints in the same circle\footnote{The graph $G_D$ can be thought as the disjoint union of the Lando graphs arising from each of the circles appearing in $s_BD$.}, and adding an edge connecting two vertices if the endpoints of the corresponding chords alternate in the same circle. See Figure \ref{hexagon}(a-d). 
\end{definition}

The independence complex $I(G)$ of a graph $G$ is defined as the simplicial complex whose simplices are the independent subsets of vertices of $G$, that is, the subsets of pairwise non-adjacent vertices of $G$. For the empty graph $G_\emptyset$ we set $I(G_\emptyset) \sim_h S^{-1}$, the sphere of dimension $-1$.

\begin{remark}\label{loop}
If $G$ contains a loop based in a vertex $v$, then $I(G)=I(G-v)$. If $e, e'$ are two edges connecting the same pair of vertices, then $I(G) = I (G-e)$. Moreover, the independence complex of the disjoint union of two graphs equals the join of their independence complexes, that is, 
 $I(G_1 \sqcup G_2) = I(G_1) \ast I(G_2)$.
\end{remark}

\begin{figure}
\centering
\includegraphics[width = 12cm]{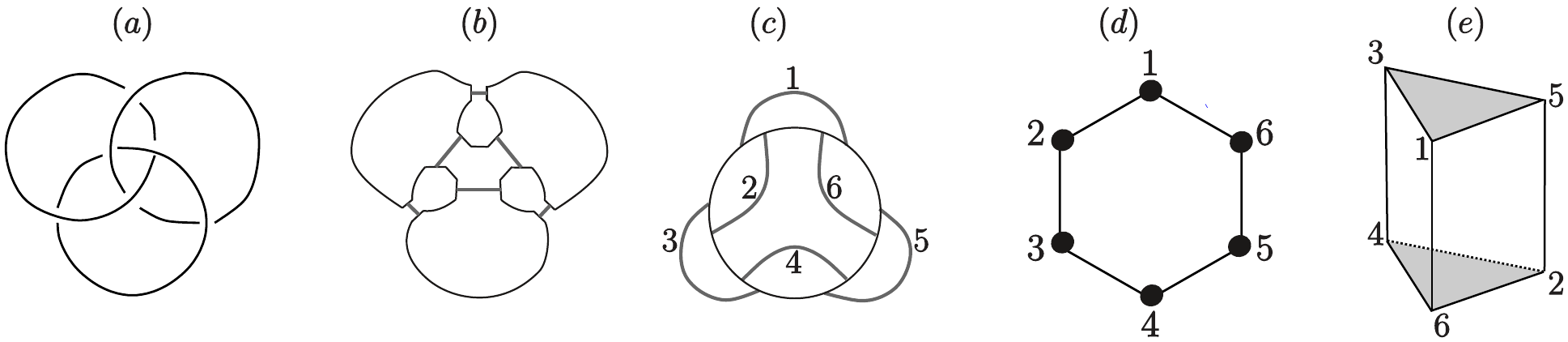}
\caption{\small{(a) A link diagram $D$; (b)-(c) the resolution $s_BD$; (d) the associated Lando graph $G_D$; (e) the geometric realization of $I_D$.}}
\label{hexagon}
\end{figure}

Given a diagram $D$, the (simplicial) complex $I_D$ is defined as the independence complex of its associated Lando graph, i.e., $I_D=I(G_D)$. Figure \ref{hexagon}(e) illustrates this definition. 

\begin{theorem}\cite{GMS}\label{teoGMS}
Let $D$ be a link diagram with $c$ crossings, and set $j=j_{\min}(D)$. Then, the chain complex $\{C_{i}(I_D), d_i\}$ is isomorphic to the extreme Khovanov complex $\{C_{i,j}(D), \partial_i\}$. In particular,
$$Kh_{2i-c,j}(D)= \tilde H_{i-1}(I_D).$$
\end{theorem}

It follows from \cite{FedSil} that the complex $I_D$ is stably homotopy equivalent to the Khovanov spectrum introduced by Lipshitz and Sarkar at its extreme quantum grading $\mathcal{X}^{j_{\min}}(D)$ (see \cite{lipshitzsarkar}). Sometimes we refer to $I_D$ as the geometric realization of the extreme Khovanov homology of $D$, and to the homotopy type of $I_D$ as the extreme Khovanov homotopy type of the diagram.

\section{Independence complexes}\label{Section 3}

\subsection{Basic results}\label{secbasic}

We summarize some results concerning independence complexes which will be useful in the next sections (see e.g. \cite{JACO}). 

Given a vertex $v$ of a graph $G$, write $N_G(v)$ for the set of adjacent vertices to $v$ in $G$ and define\footnote{These are particular cases of the most general concept of link and star of a vertex in a simplicial complex.} the star of $v$ in $G$ as  $st_G(v)=N_G(v) \cup \{v\}$.

\begin{definition}
Given two vertices $v, w$ of a graph $G$, we say that $v$ dominates $w$  if $N_G(w)\setminus \{v\} \subset N_G(v)\setminus \{w\}$. We write $v > w$.
\end{definition}

\begin{lemma}[Domination Lemma]\label{domlemma}
Let $v, w$ be two vertices of a graph $G$ such that $v$ dominates $w$. 
\begin{enumerate}
\item If $v$ and $w$ are not adjacent in $G$, then $I(G) \sim_h I(G-v)$.
\item If $v$ and $w$ are adjacent in $G$, then $I(G) \sim_h I(G-v) \vee \Sigma I(G-st(v))$, where $\vee$ denotes the wedge of two simplicial complexes.
\end{enumerate}
\end{lemma}

Domination Lemma is a consequence of the following more general result:

\begin{proposition}\label{Prop3.3}
Let $v$ ve a vertex of a graph $G$. If $I(G-st(v))$ is contractible in $I(G-v)$, then $$I(G) \sim_h I(G-v) \vee \Sigma I(G-st(v)).$$
\end{proposition}

A vertex $v\in G$ of degree one is called a \textit{leaf} and its unique adjacent vertex a \textit{preleaf}. The following result is a direct consequence of Domination Lemma.

\begin{corollary}\label{leaf}
Let $w$ be a leaf of a graph $G$ and let $v$ be its associated preleaf. Then, $$I(G) \sim_h \Sigma I(G-st(v)).$$
\end{corollary}

\begin{proposition}[Generalized Csorba]\cite{Csorba, JACO} \label{genCsorba}
Let $G$ be a graph with two vertices $v_1$ and $v_2$ connected by a path $L$ of length 3 (the two internal vertices of $L$ having degree two), and let $H$ be the graph obtained from $G$ by contracting $L$. Then $I(G) \sim_h \Sigma I(H)$. 
\end{proposition}

We recall some results about the homotopy type of the independence complex of some well-known families of graphs:

\begin{proposition}\label{Koz} \cite{kozlov} \mbox{ } 
\begin{enumerate}
\item
Let $L_n$ be the $n$-path (i.e., the line graph consisting of $n+1$ vertices connected by $n$ edges). Then
$$I(L_n) \sim_h
\left\{
  \begin{array}{cll}
     * & & \mbox{if } n=3k, \\
     S^k & & \mbox{if } n=3k+1, \, 3k+2.
  \end{array}
  \right.
  $$
\item If $T$ is a forest, then its independence complex is either contractible or homotopy equivalent to a sphere. The exact homotopy type can be found in polynomial time by repeated use of Corollary \ref{leaf}.
\item Let $P_n$ be the cycle graph of order $n$, that is, the $n$-gon. Then
$$I(P_n) \sim_h
\left\{
  \begin{array}{lll}
     S^{k-1}  & & \mbox{if } n=3k \pm 1, \\
     S^{k-1} \vee S^{k-1} & & \mbox{if } n=3k.
  \end{array}
  \right.
  $$
  \end{enumerate}
\end{proposition}

The following basic lemma on line graphs will be useful in Sections \ref{sectfan} and \ref{sectionc3}. 

\begin{lemma}\label{Lemma123} Consider the line graph $L_n$.
\begin{itemize}
    \item[(1)] Given a vertex $v$ of $L_n$,  then 
either one of $I(L_n)$ or $I(L_n-st(v))$ is contractible, or $I(L_n)\sim_h \Sigma I(L_n-st(v))$. 
\item[(2)] Given a subset $A$ of vertices of $L_n$, then either $I(L_n)$ or $I(L_n-A)$ is contractible, or $I(L_n)$ and $I(L_n-A)$ are spheres and $dim (I(L_n)) \geq dim (I(L_n-A))$.  
\end{itemize}
\end{lemma}

Another basic calculation concerns the $\theta$-graph and its subdivisions, which we address in the following example.

\begin{example}\label{exampletheta}
Consider the graph $\theta_{n_1,n_2,n_3}$ obtained from $\theta$-graph by subdividing each of its edges into $n_1$, $n_2$ and $n_3$ pieces, respectively (see Figure~\ref{thetagraph}). The problem of computing $I(\theta_{n_1,n_2,n_3})$ can be reduced by Proposition \ref{genCsorba} to a finite number of cases: $0\leq n_2 \leq 2$, $1\leq n_1,n_3 \leq 3$. Analyzing each of them we conclude that $I(\theta_{n_1,n_2,n_3})$ is either contractible or homotopy equivalent to either a sphere or the wedge of two spheres (possibly of different dimensions). 

In particular, the only case leading to the wedge of two spheres of different dimensions is when $n_1\equiv n_2 \equiv n_3 \equiv 0$ $(mod \, 3)$, since $\theta_{n_1,0,n_3} = P_{n_1}\vee P_{n_3}$ and thus $$I(\theta_{3k_1,3k_2,3k_3})\sim_h S^{k_1+k_2+k_3-1} \vee S^{k_1+k_2+k_3-2}.$$

Furthermore, observe that every proper subgraph $G$ of $\theta_{n_1,n_2,n_3}$ can be reduced by Corollary \ref{leaf} to either a forest or a polygon. Thus $I(G)$ is either contractible or it is homotopy equivalent to a sphere or a wedge of two spheres of the same dimension.

In forthcoming sections we will make use of the following computation: $$I(\theta_{4,2,4}) \sim_h \Sigma^2 I(\theta_{1,2,1}) \sim_h \Sigma^2 I(P_3)\sim_h S^2\vee S^2.$$
\end{example}

\begin{figure}
\centering
\includegraphics[width = 10.5cm]{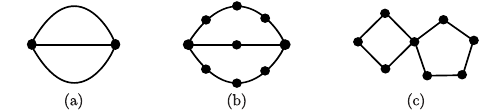}
\caption{\small{The $\theta$-graph and its subdivisions $\theta_{4,2,4}$ and $\theta_{4,0,5}$ are shown in (a), (b) and (c), respectively.}}
\label{thetagraph}
\end{figure}

In the topological part of this paper (Sections \ref{secpositive} and \ref{sectnegative}) we usually work with induced subgraphs of a given graph. However, sometimes we have to remove a special edge (keeping its endpoints). The next result allows us to do so in certain situations.

\begin{proposition}\label{trickremovingedge}
Consider a family of graphs $\mathcal G$ closed under taking induced subgraphs. Assume that for a given edge $e$ of a graph $G \in \mathcal G$ any graph obtained from $G$ by subdividing $e$ (that is replacing $e$ by $L_n$, $n\geq 1$) is also in $\mathcal G$. Then $G-e$ is also in $\mathcal G$.
\end{proposition}

\begin{proof}
Starting from $G$, subdivide $e$ into 2 pieces (i.e. replace $e$ by $L_2$) and then remove the middle vertex of $L_2$. The obtained graph is $G-e$ and it belongs to $\mathcal G$.
\end{proof}

\subsection{Fan, rhomboid and augmented rhomboid graphs}\label{sectfan}
In this section we introduce some families of graphs whose independence complexes are  crucial in Sections \ref{secpositive} and \ref{sectnegative} when analyzing (extreme) Khovanov homotopy type  associated to 4-braid diagrams.

The join of two graphs $G_1$ and $G_2$, denoted by $G_1*G_2$, is the one skeleton of the join in the category of simplicial complexes. The cone $G$ of a graph $H$ is the join of $H$ with an isolated vertex $K_1$ that we call apex of $G$, that is, $G=cone_1(H) = K_1 * H$.

\begin{lemma}\label{KjH} Let $H$ be a non-empty graph and denote by $K_j$ the complete graph on $j$ vertices. Then, $I(K_j*H) \sim_h I(H) \vee \underbrace{S^0\vee S^0\vee ...\vee S^0}_{j}$.
\end{lemma}

\begin{proof}
The result follows from the fact that vertices in $K_j$ become isolated vertices in $I(K_j*H)$. 
\end{proof}

\begin{definition}
A graph $F_n$ is called a \textit{simple fan} if $F_n = cone_1(L_n)$. A graph is said to be a \textit{fan} if it can be obtained from a simple fan by subdividing some of the edges of $L_n$. See Figure \ref{fansimple}.
\end{definition}

\begin{figure}
    \centering
\includegraphics[width = 11cm]{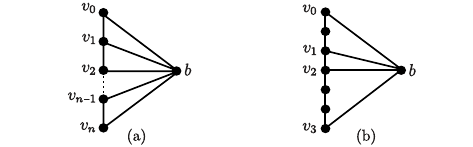}
\caption{\small{The simple fan $F_n$ and a fan obtained from the simple fan $F_3$ are shown in (a) and (b), respectively.}}\label{fansimple}
\end{figure}

It follows from Lemma \ref{KjH} that $I(F_n) \sim_h I(L_n) \vee S^0$.

\begin{proposition}\label{Propfan} Let $G$ be an induced subgraph of a fan. Then $I(G)$ is either contractible or homotopy equivalent to the wedge of at most two spheres. \
\end{proposition}

\begin{proof}
If $b\notin G$, then $G$ is a forest and the result follows from Proposition \ref{Koz}(2). Thus, assume that $G$ contains $b$ and it has no leaves (otherwise apply Corollary \ref{leaf}). If $I(G-st(b))$ is contractible, then by Proposition \ref{Prop3.3} $I(G)~\sim_h~I(G-b)$, with $G-b$ a disjoint union of paths. Otherwise, $G$ can be reduced by a finite number $k$ of Csorba moves (Proposition \ref{genCsorba}) to $cone_1(H)$, with $H$ a disjoint union of paths. Thus $I(G) \sim_h \Sigma^k (I(H) \vee S^0$) and Proposition \ref{Koz} completes the proof.   
\end{proof}




\begin{definition}\label{defallrhomboids} \ 
\begin{enumerate}
\item Consider $S^0$ with vertices $b_{\ell}$ and $b_r$. The {\it{simple rhomboid graph}} $SR_n$ is obtained from the join $S^0*L_n$ after adding four paths of length $2$ connecting each $b_{\ell}$ and $b_r$ with each of the two vertices of degree 1 in $L_n$. Labeling of the vertices from Figure \ref{defrhomboids}(a) will be used throughout the paper. 
\item The {\it{simple connected rhomboid graph}} $SR^{con}_n$ is obtained from $SR_n$ by adding an edge connecting $b_{\ell}$ and $b_r$, as shown in Figure \ref{defrhomboids}(a) - including grey edge.
\item The {\it{simple augmented rhomboid graph}} $SR_n^{aug}$ is constructed from $SR_n - \{v_{-1}', v_{n+1}'\}$ by adding vertices $b_2$, $c_i$ and $d_i$, for $i=1,2,3$, with connections as defined in Figure \ref{defrhomboids}(c). 
\end{enumerate} 
Given a graph $G$ as above, vertices $v_i$ are called spoke vertices, for $1 \leq i \leq n$, and the $n$-path connecting them is the {\it{spine}} of $G$. The {\it{extended spine}} of $G$ consist of its spine together with edges connecting $v_0$ and $v_{-1}$ and $v_n$ and $v_{n+1}$. If $H$ is an induced subgraph of $G$, then $H^{spine}$ denotes the intersection of $H$ and the spine of $G$.  \end{definition}

\begin{figure}
    \centering
\includegraphics[width = 11cm]{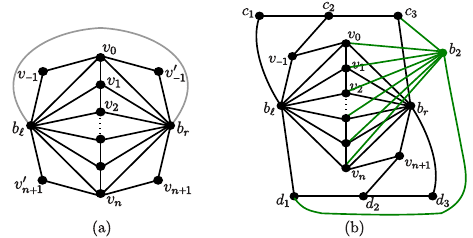}
\caption{\small{The simple rhomboid $SR_n$ is shown in black in (a). If we additionally consider the grey edge connecting $b_\ell$ and $b_r$, then we get the connected rhomboid $SR_n^{con}$. The simple augmented rhomboid $SR_n^{aug}$ is shown in (b).}}\label{defrhomboids}
\end{figure}

We next analyze the independence complex of each of the graphs defined above and these of their induced subgraphs.

\begin{proposition}\label{SimRhomboid}
Let $G$ be an induced subgraph of either a simple rhomboid $SR_n$ or a simple connected rhomboid $SR_n^{con}$ and let\\ $X=\{v_{-1}, v_{-1}', v_{n+1}, v_{n+1}'\}$. Then:

\begin{enumerate}
\item[(1)]
If $G^{spine}=\emptyset$ then $I(G)$ is either contractible or homotopy equivalent to the spheres $S^{-1}, S^0$ or $S^1$.

\item[(2)] If $G\cap X\neq \emptyset$ then $I(G)$ is either contractible or has the homotopy type of a sphere or two spheres. Moreover, if we assume without loss of generality that $v_{-1} \in G$, then: 
\begin{itemize}
    \item[(a)] If either $b_\ell$ or $v_0$ (but not both) is in $G$, then $I(G)$ is contractible or homotopy equivalent to a sphere.
    \item[(b)] If neither $b_\ell$ nor $v_0$ are in $G$, then $I(G)$ is contractible.
\end{itemize}

\item[(3)]
If $G\cap X = \emptyset$ and $G^{spine} \neq \emptyset$, then:
\begin{itemize}
    \item[(a)] If $G$ is an induced subgraph of $SR_n$ and $b_\ell, b_r \in G$, then $$I(G)\sim_h I(G\cap F_n)\sim_h I(G \cap L_n )\vee S^0.$$
    \item[(b)] If $G$ is an induced subgraph of $SR_n^{con}$ and $b_\ell, b_r \in G$, then $$I(G) \sim_h I(G\cap L_n)\vee S^0 \vee S^0.$$
    \item[(c)] If $b_\ell$ or $b_r$ is not in $G$, then $I(G)$ is either contractible or has the homotopy type of either a sphere or the wedge of two spheres.
\end{itemize}
\end{enumerate}

In particular, $I(G)$ can be homotopy equivalent to a wedge of three spheres only in the case {$\operatorname{(3 b)}$}; moreover, such a wedge is of type $S^j\vee S^0\vee S^0$, with $j\geq 0$. Otherwise $I(G)$ is either contractible or homotopy equivalent to either a sphere or the wedge of two spheres.
\end{proposition}

\begin{proof} We use Domination Lemma (Lemma \ref{domlemma}) and Corollary \ref{leaf} in the proof. \\
(1) It is clear from Figure \ref{defrhomboids} that $G$ is a forest with at most 5 edges and therefore the result holds. \\ 
(2) If $v_{-1}, b_\ell$ and $v_0$ are in $G$, then $b_\ell$ dominates $v_{-1}$, and $v_{-1}$ becomes a leaf in $G-b_\ell$, so $$I(G) \sim_h \Sigma I(G-st(b_\ell)) \vee \Sigma I(G-st(v_0)),$$ and these two graphs are forests and Proposition \ref{Koz}(2) applies. 
In the particular case (a) (resp. (b)) the vertex $v_{-1}$ is a leaf (resp. isolated) and the statement holds.  \\
(3) Case (a) follows directly from domination of $b_\ell$ over $b_r$. Case (b) follows from Lemma \ref{KjH} and the fact that $G=K_2*G^{spine}$. In case (c) $G$ is an induced subgraph of a fan and Proposition \ref{Propfan} completes the proof. 
\end{proof}

Next, we determine the homotopy type of the independence complex of the simple augmented rhomboid graph $SR_n^{aug}$ and those of their induced subgraphs. 

The case when the induced subgraph $G$ does not contain all $b_2, b_\ell$ and $b_r$ is simpler, since it can be easily reduced to an induced subgraph of $SR_n$ or $SR_n^{con}$. In particular, if $G=SR_n^{aug}-b_2$, applying twice Domination Lemma ($v_{-1}>c_1$ and $v_{n+1} > d_3$) and twice Csorba reduction (Proposition \ref{genCsorba}), we get $SR_n^{con}$. Analogous results hold for the cases $SR_n^{aug}-b_\ell$ and $SR_n^{aug}-b_r$.

We will show that, up to some minor conditions, if the induced subgraph $G$ contains $b_2, b_\ell$ and $b_r$ then $I(G)$ can be obtained by recursive application of Proposition \ref{Prop3.3}. More precisely, we will check the conditions when {\small{$$I(G)\sim_h  I(G-b_2-b_{\ell}-b_r) \vee \Sigma I(G-b_2-b_{\ell}- st(b_r)) \vee \Sigma I(G-b_2- st(b_{\ell})) \vee \Sigma I(G-st(b_2)),$$}}where each term can be computed easily, leading to either a sphere or contractible homotopy type. In particular, we get that $I(G)$ has the homotopy type of four (resp. three) spheres if all four (resp. three) terms are not contractible.  

\begin{proposition}\label{propaugmented}
Let $G$ be an induced subgraph of $SR_n^{aug}$ containing vertices $b_2, b_{\ell}$ and $b_r$, and let $G^{spine}\neq \emptyset$. Then:
\begin{enumerate}
\item[(1)] Either $I(G)$ is contractible or $$ I(G) \sim_h  I(G-b_2)\vee \Sigma I(G-st(b_2)).$$
\item[(2)] Either $I(G)-b_2$ is contractible or $$ I(G-b_2) \sim_h  I(G-b_2-b_{\ell})\vee \Sigma I(G-b_2-st(b_{\ell})).$$
\item[(3)] Either $I(G-b_2-b_{\ell})$ is contractible or
$$ I(G-b_2-b_{\ell}) \sim_h  I(G-b_2-b_{\ell}-b_r) \vee \Sigma I(G-b_2-b_{\ell}-st(b_r)).$$
\end{enumerate}
\end{proposition}

\begin{proof} 
The statement follows from iterative application of Proposition \ref{Prop3.3}. In each step one needs to check that $I(G-st(v))$ is contractible in $I(G-v)$, for $v\in\{b_2, b_\ell, b_r\}$. These conditions can be deduced from Lemma \ref{lemaaugmented}.
\end{proof}

\begin{remark}
The case when $G$ is an induced subgraph of $SR_n^{aug}$ with $G^{spine}= \emptyset$ was discussed in Example \ref{exampletheta}, since after possible application of Domination Lemma ($v_{-1}>c_1$ and $v_{n+1}>d_3$), $G$ becomes an induced subgraph of $\theta_{4,2,4}$.
\end{remark}

\begin{lemma}\label{lemaaugmented}
Let $G$ be an induced subgraph of $SR^{aug}_n$ with $G^{spine}\neq \emptyset$ and $b_{\ell},b_r,b_2 \in G$. Then the following conditions hold: 
\begin{enumerate}
\item[(1)] $$I(G-st(b_{2}) )\sim_h
\left\{
\begin{array}{ll}
S^1 & \mbox{if ($c_1$ or $v_{-1}$) and  ($d_3$ or  $v_{n+1}$) are in $G$},\\
\ast   & \mbox{otherwise}.
\end{array}
\right.
$$

\item[(2)] $$I(G-b_2-st(b_{\ell}))\sim_h
\left\{
\begin{array}{ll}
S^1 & \mbox{if $c_2,c_3,d_2$ and ($d_3$ or $v_{n+1}$) are in $G$},\\
S^0 & \mbox{if (($c_2,c_3\in G$) and ($d_2,d_3,v_{n+1}\notin G$)) or} \\
        & \mbox{ (($c_2,c_3 \notin G$) and ($d_2$ and ($d_3$ or $ v_{n+1}$) are in $G$)) or }\\
& \mbox{ (($c_2,d_2 \notin G$) and ($c_3$ or $d_3$ or $v_{n+1}$ is in $G$))},\\
\ast   & \mbox{otherwise}.
\end{array}
\right.
$$

\item[(3)]
$$I(G-b_2-b_{\ell}-st(b_r))\sim_h
\left\{
\begin{array}{ll}
S^1 & \mbox{if $d_1,d_2,c_2$ and ($c_1$ or $v_{-1}$) are in $G$},\\
S^0 & \mbox{if (($d_1,d_2 \in G)$ and ($c_1,c_2,v_{-1} \notin G$))  or} \\
& \mbox{(($d_1,d_2 \notin G$) and ($c_2$ and ($c_1$ or $v_{-1}$) are in $G$))},\\
        S^{-1} & \mbox{if  $G-b_2-b_{\ell}-st(b_r)=\emptyset$},\\
\ast   & \mbox{otherwise}.
\end{array}
\right.
$$

\item[(4)] Assume that $I(G-~b_2-~b_{\ell})$ is not contractible. Then, $I(G-~b_2-~b_\ell~-~b_r)$ is either contractible or homotopy equivalent to a sphere $S^j$ for some $j$. In particular, for the latter case we get:  $$I(G-b_2-b_{\ell}-b_r)\sim_h
\left\{
\begin{array}{ll}
S^j, \, j>1 & \mbox{if ($c_2$ and ($c_1$ or $c_3$ or $v_{-1}$)) and } \\
& \mbox{($d_2$ and ($d_1$ or $d_3$ or $v_{n+1}$)) are in $G$},\\
S^j, \, j>0 & \mbox{if (($c_2$ and ($c_1$ or $c_3$ or $v_{-1}$) are in $G$) or } \\
& \mbox{($d_2$ and ($d_1$ or $d_3$ or $v_{n+1}$) are in $G$)},\\
S^j, \, j\geq 0  & \mbox{otherwise}.
\end{array}
\right.$$
\end{enumerate}
\end{lemma}

\begin{figure}
    \centering
\includegraphics[width = 11cm]{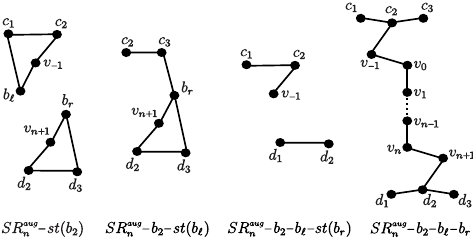}
\caption{\small{Graphs illustrating proof of Lemma \ref{lemaaugmented} in the case when $G=SR_n^{aug}$.}}\label{inducedaugmented}
\end{figure}

\begin{proof}
In Figure \ref{inducedaugmented} we illustrate the graphs $SR_n^{aug}-st(b_2)$, $SR_n^{aug}-b_2-st(b_\ell)$, $SR_n^{aug}-b_2-b_\ell-st(b_r)$, and $SR_n^{aug}-b_2-b_\ell-b_r$. Statements (1), (2) and (3) follow from the analysis of Figure \ref{inducedaugmented}. To get statement (4) we analyze all different possibilities for $G^{spine}$ and apply Lemma~\ref{Lemma123}.
\end{proof}

Proposition \ref{propaugmented} and careful analysis of Lemma \ref{lemaaugmented} allow us to fully characterize the induced subgraphs of $SR^{aug}_n$ whose independence simplicial complexes are homotopy equivalent to a wedge of three or four spheres. 

\begin{corollary}\label{corollarydimensions}
Let $G$ be an induced subgraph of $SR_n^{aug}$ with $G^{spine} \neq \emptyset$. and let $A^+,A^-,B^+,B^-$ describe the following conditions on $G$:\\
$A^+$ means that $c_2,c_3$ and ($c_1$ or $v_{-1}$) are in $G$.\\
$A^-$ means then none of $c_2,c_3,c_1,v_{-1}$ are in $G$.\\
$B^+$ means that $d_1,d_2$ and ($d_3$ or $v_{n+1}$) are in $G$.\\
$B^-$ means then none of $d_1,d_2,d_3,v_{n+1}$ are in $G$.
\begin{enumerate}
\item[(1)] $I(G)$ is a wedge of four spheres if and only if  $b_2, b_{\ell},b_r$ are in $G$, conditions $A^+$ and $B^+$ hold,
and $I(G^{spine})$ is not contractible (that is $I(G^{spine}) \sim_h S^j$ for some $j$). Furthermore, we get $$I(G) \sim_h S^j \vee S^2 \vee S^2 \vee S^2,$$ where $j\geq2$.
\item[(2)] $I(G)$ is a wedge of three spheres if and only if  one of the following conditions holds: \\
(i) $b_{\ell}, b_r$ are in $G$, $b_2 \notin G$, conditions $A^+$ and $B^+$ hold and $I(G^{spine})$ is not contractible.\\
(ii) $b_r$ and ($b_{\ell}$ or $b_2$) are in $G$, conditions $A^-$ and $B^+$ hold and $I(G^{spine})$ is not contractible.\\
(iii) $b_{\ell}$ and ($b_r$ or $b_2$) are in $G$, conditions $A^+$ and $B^-$ hold and $I(G^{spine})$ is not contractible.\\
(iv) $b_2, b_{\ell},b_r$ are in $G$, conditions $A^+$ and $B^+$ hold and $I(G^{spine})$ is contractible.
\end{enumerate}
Furthermore, we get $I(G) \sim_h S^j \vee S^2 \vee S^2$, where $j\geq2$.
\end{corollary}

The property that one of the spheres in Corollary \ref{corollarydimensions} has equal or higher dimension than the others (i.e., $j\geq 2$) can be proved using Lemma \ref{Lemma123}.

We present now two crucial examples:

\begin{example}\label{Example 3.17} For the simple augmented rhomboid graph $SR^{aug}_n$ (see Figure \ref{defrhomboids}(b)) we get 
$$I(SR_n^{aug}) \sim_h \left\{
\begin{array}{ll}
S^2 \vee S^2 \vee S^2 & \mbox{if $n=3k$},\\
S^{k+2} \vee S^2 \vee S^2 \vee S^2& \mbox{if $n=3k+1$ or $n=3k+2$}.
\end{array}
\right.
$$
To obtain the formula above, observe first that since $v_{-1} > c_1$ and $v_{n+1}>d_3$ Domination Lemma applies and $I(SR_n^{aug})\sim_h I(SR_n^{aug}-v_{-1}-v_{n+1})$. Then applying Csorba reduction (Proposition \ref{genCsorba}) over the paths $b_\ell-c_1-c_2-c_3$ and $d_1-d_2-d_3-b_r$ we get $I(SR_n^{aug}-v_{-1}-v_{n+1})\sim_h \Sigma^2 I(K_3*L_n)$ and Lemma \ref{KjH}(2) and Proposition \ref{Koz}(1) completes the proof.
\end{example} 

\begin{example}\label{Example 3.18} Let $G=SR^{aug}_n-c_1-d_3$.
Then the homotopy type of $I(G)$ is the same as in the previous example, that is 
$$I(SR^{aug}_n-c_1-d_3)\sim_h I(SR_n^{aug}).
$$
To get the above result, observe that $I(G-st(b_2)) \sim_h I(L_2 \sqcup L_2) \sim_h S^1$ and after applying twice Csorba reduction we get $I(G-b_2) \sim_h \Sigma^2 (K_2*L_n)$, which has the homotopy type of either $S^2 \vee S^2$ (if $n=3k$) or $S^{k+2} \vee S^2 \vee S^2$ (if $n=3k+1, 3k+2$). Therefore, Proposition \ref{Prop3.3} applies and we get 
$$I(G) \sim_h I(G-b_2) \vee \Sigma I(G-st(b_2)).$$
\end{example}

In the topological part of this paper (Sections \ref{secpositive} and \ref{sectnegative}) we need to consider induced subgraphs of more general augmented rhomboids (see Definition~\ref{nosimpledef}); essentially we allow subdivision of edges $\overline{v_iv_{i+1}}$, $-1\leq i \leq n$. We also consider certain modifications of simple rhomboid graphs. 

\begin{definition}\label{modaugmented}
We define the family of mod-augmented simple rhomboid graphs as those graphs obtained from an augmented simple rhomboid graph $SR_n^{aug}$ after (possibly) a combination of the following transformations:
\begin{enumerate}
    \item[(a)] contract edge connecting $v_0$ and $v_{-1}$,
    \item[(b)] contract edge connecting $v_n$ and $v_{n+1}$,
    \item[(c)] delete edge connecting $v_0$ and $v_{-1}$,
    \item[(d)] delete edge connecting $v_n$ and $v_{n+1}$.
\end{enumerate}
\end{definition}

We can analyze independence complexes of induced subgraphs $G$ of mod-augmented simple rhomboid graphs in the same way we did with induced subgraphs of $SR^{aug}_n$. In fact, in most cases there is no need to repeat the whole proof, as those induced graphs obtained from modifications of $SR_n^{aug}$ can be reduced to induced subgraphs of $SR^{aug}_k$ with $k\leq n$, as illustrated in the following example. 

\begin{figure}
\includegraphics[width = 10.5cm]{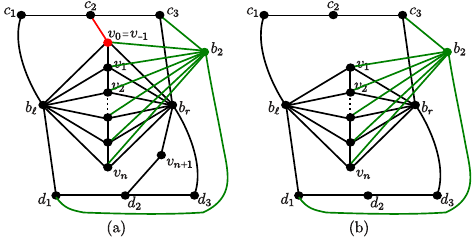}
\caption{\small{Modified augmented simple graph $G$ illustrating Example \ref{exammodifiedaug} and its reduction $SR_k^{aug}-v_{-1}-v_{k+1}$, with $k=n-1$.}}
\label{figmodifiedaug}
\end{figure}

\begin{example}\label{exammodifiedaug}
Let $G$ be the graph obtained from $SR_n^{aug}$ after collapsing edge connecting $v_0$ and $v_{-1}$ and removing edge connecting $v_n$ and $v_{n+1}$, as shown in Figure \ref{figmodifiedaug}(a). It is clear that $v_0 > c_1$ and $v_{n+1}> d_3$ in $G$, and therefore by Domination Lemma $I(G) \sim_h I(G-v_0-v_{n+1})$. The graph $G-v_0-v_{n+1}$ is the graph $SR_k^{aug}-v_{-1}-v_{k+1}$ where $k=n-1$, as shown in Figure \ref{figmodifiedaug}(b).
\end{example}

Next result summarizes the computations described before. 

\begin{corollary}\label{cormodaugmented4spheres}
Let $G$ be an induced subgraph of a (mod-)augmented simple rhomboid graph. Then, there exists a polynomial time algorithm which determines the homotopy type of $I(G)$. Moreover, if not contractible, $I(G)$ is homotopy equivalent to either $S^i$ or $S^k\vee S^i$ or $S^k\vee S^i \vee S^i$ or $S^k\vee S^i\vee S^i \vee S^i$, where $k\geq i$. 
\end{corollary}

\begin{proof}
Homotopy types follow from earlier results in this section. Moreover, all constructions in this section depended only on tasks performed in polynomial time (often linear or quadratic), for example finding a preleaf in the graph or checking whether one specific vertex dominates other.
\end{proof}

We finish this section introducing a family of graphs which will be obtained as Lando graphs of the closures of certain braids of 4 strands. These graphs are obtained from (augmented) simple rhomboid graphs by allowing subdivisions in the extended spine.

\begin{definition}\label{nosimpledef}
The rhomboid graph $G(c_0',c_0,\ldots,c_{n+1},c_{n+1}')$ is the graph obtained from the simple rhomboid graph $SR_n$ (Definition \ref{defallrhomboids}) by subdividing the edge connecting $v_{i-1}$ to $v_i$ into $c_i$ parts, for $0\leq i \leq n+1$, and the edges connecting $v_{-1}'$ to $v_0$ and $v_n$ to $v_{n+1}'$ into $c_0'$ and $c_{n+1}'$ parts, respectively. See Figure \ref{exrhomboidgraph}. 

If $c_i=0$ for some $i$, then vertices $v_{i-1}$ and $v_i$ are identified (similarly for $c_0'=0$ and $c_{n+1}'=0$), and we remove multiple edges. 

Analogously, we define (mod-)augmented rhomboid graph as those graphs obtained by subdivision along the extended spine of a (mod-)augmented simple rhomboid graph.
\end{definition}

\begin{figure}[h]
\includegraphics[width = 11cm]{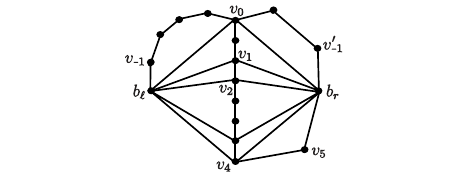}
\caption{\small{Rhomboid graph $G(2,4,2,1,3,1,1,0)$. Since $c_5'=0$, vertices $v_5'$ and $v_4$ were identified.}} 
\label{exrhomboidgraph}
\end{figure}        

\begin{proposition}\label{rhomboid2spheres}
Let $G=G(c_0',c_0,c_1,...,c_{n+1},c_{n+1}')$ be a rhomboid graph. Then the homotopy type of the independence complex of any induced subgraph $H$ of $G$ can be computed in polynomial time, and $I(H)$ is either contractible or homotopy equivalent to a sphere or to the wedge of two spheres. 
\end{proposition}

\begin{proof}
We introduce the notation $c_i=\infty$ if $v_{i-1}$ is not connected to $v_i$ in $H$, for $0\leq i \leq n+1$ (analogous convention is used for $c_0'$ and $c_{n+1}'$). Also, write $C_I=\{c_1, \ldots, c_n\}$ and $C_J= \{c_0, c_0', c_{n+1}, c_{n+1}'\}$.

We use Domination Lemma (Lemma \ref{domlemma}) and Csorba reduction (Proposition \ref{genCsorba}) in the proof. We also apply Corollary \ref{leaf} if possible, so we can assume that there are no leaves at any step of the process. 

First step is to apply Csorba reduction to $c_i \in C_I$, so we can assume that these parameters are equal to $0, 1, 2$ or $\infty$. We keep notation $H$ for the reduced graph. 

If $c_i=2$ for some $c_i \in C_I$, then $b_\ell$ and $b_r$ dominates the vertex between $v_{i-1}$ and $v_i$, so $I(H) \sim_h I(H-b_\ell-b_r)$ and Proposition \ref{Koz}(2) completes the proof.

Otherwise we can assume that $c_i \in C_I$ equals $0, 1$ or $\infty$. Now, we apply Csorba reduction to each element in $C_J$, and reduce them to $0, 1, 2$ or $\infty$ .

If $c_0=2$, then $v_0 > v_{-1}$ and $v_{-1}$ becomes a preleaf in $G-v_0$, so $I_H \sim_h \Sigma I(H-S)$, where $S$ is the square containing vertices $b_l,v_{-1}, v_0$ and an additional vertex. The graph $H-S$ is an induced subgraph of a fan, so Proposition \ref{Propfan} completes the proof. Analogous result holds for the case when other element in $C_J$ equals $2$.

Otherwise, $H$ is an induced subgraph of a simple rhomboid graph, and Proposition \ref{SimRhomboid} completes the proof. 
\end{proof}

\section{Braid and Temperley-Lieb diagrams}\label{Section 5}

The basic objects in this paper are $n$-braid diagrams which, from an algebraic point of view, are elements (words) of the free monoid generated by $2n-2$ letters $$M_n=\{\sigma_1,\sigma_1^{-1},\sigma_2,\sigma_2^{-1},\ldots, \sigma_{n-1},\sigma_{n-1}^{-1}\}$$ with the identity element $\varepsilon$ representing the empty word and geometric interpretation of generators as shown in Figure \ref{figgene}(a). We stress that braid group relations\footnote{Presentation of the braid group on $n$ strands was given by Artin \cite{Ar1}: \\ $B_n=\{\sigma_1,...,\sigma_{n-1} \ | \ \sigma_i\sigma_{i+1}\sigma_i=\sigma_{i+1}\sigma_i\sigma_{i+1}, \,  \sigma_i\sigma_j=\sigma_j\sigma_i \mbox{ for } |i-j|>1\}$.} do not hold in $M_n$. Define the submonoids $M_n^+ = \{\sigma_1, \ldots, \sigma_{n-1}\}\subset M_n$ and $M_n^- = \{\sigma_1^{-1}, \ldots, \sigma_{n-1}^{-1}\}\subset M_n$. Given a word $w\in M_n$, we write $w^+ \in M_n^+$ for the word obtained from $w$ by deleting those letters with negative exponents. We say that $w$ is positive if $w = w^+ \in M_n$. A subword\footnote{Note that a subword of a cyclic word is not a cyclic word.} of $w$ is the word consisting of some consecutive letters of $w$. 

\begin{figure}
    \centering
    \includegraphics[width = 9cm]{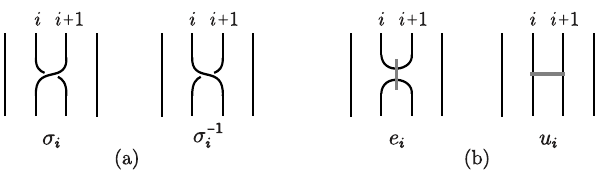}
\caption{\small{Geometric interpretation of the generators of the monoid $M_n$ (a) and the monoid $M_n^{TL}$ (b).}}\label{figgene}
\end{figure}

Define $\tilde M_n$ as the quotient of $M_n$ modulo cyclic permutation of letters in a word. Given a word (braid diagram) $w\in M_n$, we write $\tilde{w} \in \tilde{M}_n$ for the associated cyclic word (closed braid diagram, which is also a link diagram). 

Consider now the free monoid $$M^{TL}_n=\{e_1,u_1,e_2,u_2,...,e_{n-1},u_{n-1}\},$$ whose elements are called Temperley-Lieb diagrams. The geometric interpretation of $M_n^{TL}$, shown in Figure \ref{figgene}(b), is motivated by that of the original Temperley-Lieb algebra\footnote{Temperley-Lieb algebra was introduced in \cite{TL}, with formal algebraic definition by Baxter and Jones \cite{Bax, Jon1}: in the ring $\mathbb{Z}[d]$, $TL_n=  \{e_1,...,e_{n-1} \, | \, e_i^2=de_i, \, e_ie_{i{\pm 1}}e_i = e_i, \, e_ie_j=e_je_i \mbox{ for } |i-j|\geq 2\}$. Its geometric interpretation, which we use, was introduced by Kauffman \cite{Kau3,KL}.}. As before, we define $\tilde M^{TL}_n$ to be the quotient of $M_n^{TL}$ modulo cyclic permutation. 

There is an isomorphism $\varphi: M_n \to M^{TL}_n$ with $\varphi(\sigma_i)=e_i$ and $\varphi(\sigma_i^{-1})=u_i$. Domain and codomain of $\varphi$ have different geometric interpretations: given a braid diagram $w \in M_n$, $\varphi(w) \in M^{TL}_n$ is a crossingless tangle diagram with additional chords. 

We say that $\varphi(\tilde w) = \widetilde{(\varphi(w))}$ is the chord diagram of $w$. Observe that, if we write $D=\tilde{w}$, then $\varphi(\tilde{w})$ equals $s_BD$ from Definition \ref{defLando} (compare to Figure \ref{braid3}).  We write $G(w)$ for the Lando graph associated to $\varphi(\tilde{w})$ and $I(w)$ for its associated independence complex. 

\begin{figure}
    \centering
    \includegraphics[width = 11cm]{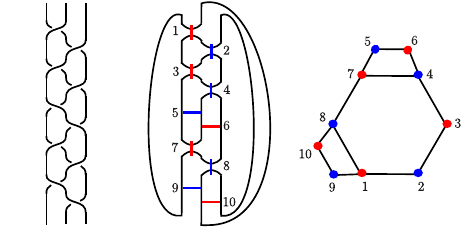}
\caption{\small{Braid diagram $w=\sigma_1\sigma_2\sigma_1\sigma_2\sigma_1^{-1}\sigma_2^{-1}\sigma_1\sigma_2 \sigma_1^{-1}\sigma_2^{-1}$ in $M_3$, its image $\varphi(\tilde{w})=e_1e_2e_1e_2u_1u_2e_1e_2u_1u_2 \in \tilde{M_3}^{TL}$ and the associated Lando graph $G(w)$, with $I(w) \sim_h S^2$.}}\label{braid3}
\end{figure}

Braid relations are not allowed in $M_n$ because, in general, they do not preserve $G(w)$ neither the homotopy type of $I(w)$. In particular, second and third Reidemeister moves can lead from a contractible independence complex to a non-contractible one. 

\begin{example}\label{exreidemeister}
\begin{enumerate}
\item Let $w_1 = (\sigma_1\sigma_2\sigma_1)^2$ and $w_2= (\sigma_1 \sigma_2)^3$ be two braid diagrams in $M_3$. It is clear that $w_1$ and $w_2$ represent the same braid in the braid group. However, $G(w_1)$ consists of two isolated vertices and therefore $I(w_1)$ is contractible, while $G(w_2)$ is an hexagon and $I(w_2) \sim_h S^1 \vee S^1$.

\item The braid words $w_1 = \sigma_1\sigma_1^{-1}\sigma_2\sigma_1\sigma_1^{-1}\sigma_2\sigma_1, \, \, w_2= \sigma_2\sigma_1\sigma_1^{-1}\sigma_2\sigma_1$ \, $w_3~=~ \sigma_2^2\sigma_1$ are related by Reidemeister-2 moves (i.e., cancellations of $\sigma_1 \sigma_1^{-1}$). However, $I(w_1) \sim_h S^1$, while $I(w_2)$ and $I(w_3)$ are contractible. 
\end{enumerate}
\end{example}

Observe that negative third Reidemeister move (i.e., $\sigma_i^{-1}\sigma_{i+1}^{-1}\sigma_i^{-1} \, \leftrightarrow \,$ $\sigma_{i+1}^{-1}\sigma_i^{-1}\sigma_{i+1}^{-1}$) and relation $\sigma_i^{\epsilon} \sigma_j^{\delta} \leftrightarrow \sigma_j^{\delta}\sigma_i^{\epsilon}$ where $|i-j|>1$ and $\epsilon, \delta \in \{1,-1\}$ preserve Lando graph. However, we decided not to incorporate them in the definition of $M_n$ as it could lead to confusion in Sections \ref{secpositive} and \ref{sectnegative}, where we consider positive braid diagrams allowing the possibility of decorating them with negative letters.

\section{Main result and outline of its proof}\label{Section 6}

The main result in this paper allows to fully characterize in polynomial time the extreme Khovanov homotopy type of closed 4-braid diagrams: 

\begin{theorem}\label{theomain}
Let $w$ be a braid diagram on $4$ strands and write $\hat{w}$ for its closure. Then, the homotopy type of the geometric realization $I(w)$ of the extreme Khovanov homology of $\hat{w}$ can be computed in polynomial time. Moreover, if not contractible, $I(w)$ is homotopy equivalent to either a sphere, or a wedge of two spheres, or $S^k \vee S^i \vee S^i$, or $S^k \vee S^i \vee S^i \vee S^i$, with $k \geq i$. 
\end{theorem}

As a consequence of the proof we get the following result, which we prove in Section \ref{secpositive}:

\begin{corollary}\label{indepeasy}
Given a positive braid diagram $w \in M_4^+$, the homotopy type of the geometric realization $I(w)$ of the extreme Khovanov homology of $\hat{w}$ can be computed in polynomial time. Moreover, $I(w)$ is either contractible or has the homotopy type of either a sphere or the wedge of two spheres.
\end{corollary}

The first step in the proof of Theorem \ref{theomain} is to define a class of braid diagrams that we call strongly-reduced (Definition \ref{defstronglyred}) and give a polynomial time algorithm that, given a braid diagram $w \in M_n$, either determines that $I(w)$ is contractible, or reduces\footnote{$w^{red}$ is obtained from $w$ by erasing some of its letters.} $w$ to a strongly-reduced braid diagram $w^{red}$ so that there exists an induced subgraph $H$ of $G(w^{red})$ with the property that $I(w)$ is homotopy equivalent to the $k^{th}$-suspension of $I(H)$ for some $k\geq 0$ (Corollary \ref{Reduction}). This result allows us to restrict our study to  strongly-reduced 4-braid diagrams.

Notice that a $B$-smoothing at a positive generator produces a horizontal smoothing with a vertical chord, while negative generators produce vertical smoothings with horizontal chords. Therefore, the \textit{shape} of the circles in the $B$-smoothed diagram is determined by positive generators in $w$. For this reason, we first study Lando graphs arising from positive braid diagrams (Section \ref{secpositive}). To do so, we classify strongly-reduced positive braid diagrams into six families, $\mathcal{C}_0 - \mathcal{C}_5$ (Proposition \ref{B4classification}). We later extend this classification to strongly-reduced braid diagrams. 

Next step is to study how negative letters of $w$ contribute to $G(w)$, i.e., how the graph $G(w^+)$ is extended when adding negative letters to $w^+$ to obtain $w$. This addition can be done one at a time, since each negative letter corresponding to an admissible (horizontal) chord adds one vertex together with some edges to the Lando graph $G(w^+)$. 

Strongly-reduced 4-braid diagrams were classified into six families $\{\mathcal{C}_i\}_{i=0}^{5}$. In Proposition \ref{easycases} we determine $G(w)$ and the homotopy type of $I(w)$ for $w \in \mathcal{C}_i$ with $i= 0, 1, 2, 5$. The cases when $w$ belongs to $\mathcal{C}_3$ or $\mathcal{C}_4$ require more effort, and they are addressed in Section \ref{sectionc3}. In order to study the effect of each negative letter of $w$ in $G(w)$, we divide $\tilde w$ into two parts, that we call \textit{head} and \textit{tail} (Definition \ref{defheadtail}).

We first prove Theorem \ref{theomain} for the family of braid words having positive tail (Proposition \ref{subwordofmaximalhead}). Finally, in Section \ref{secttail} we deal with those negative letters appearing in the tail, showing that they can be eliminated (Proposition \ref{reducetopositivetail}) and reducing this case to the previous one. Working with negative letters in the head of $w$ leads to \textit{augmented} and \textit{mod-augmented rhomboid graphs} (see Definitions \ref{modaugmented} and \ref{nosimpledef}). Determining the homotopy type of the independence complexes of these families of graphs and those of their induced subgraphs is crucial in our work. This was analyzed in Section \ref{sectfan}.

\section{Strongly-reduced braid diagrams}\label{secstronglyred}

In this section we introduce the set of strongly-reduced $n$-braid diagrams $M_n^{red}$ and show that any braid diagram $w$ can be reduced in polynomial time to an element $w^{red}\in M_n^{red}$ whose Lando graph $G(w^{red})$ allows to compute $I(w)$ in polynomial time. 

\begin{definition}\label{defstronglyred} Let $w\in M_n$.
\begin{enumerate}
\item[(a)] $w$ is positive square free if $\tilde w$ does not contain a subword $\sigma_iw_1\sigma_i$ so that $\sigma_j \notin w_1$ for
$|i-j| \leq 1$.
\item[(b)] $w$ is negative square free if $\tilde w$ does not contain a subword $\sigma_i^{-2}$ or $\sigma_i^{-1}\sigma_{i\pm 1}^{-1}\sigma_i^{-1}$. 
\item[(c)] $w$ is nesting free if $\tilde w$ does not contain a subword $\sigma_iw_1\sigma_i^{-1}w_2\sigma_i^{-1}$ or $\sigma_i^{-1}w_2\sigma_i^{-1}w_1\sigma_i$, where $\sigma_j$ is neither in $w_1$ nor $w_2$ for $|i-j|\leq 1$.
\item[(d)] $w$ is $R_2$-reduced if $\tilde w$ does not contain a subword $\sigma_iw_1\sigma_i^{-1}$ or $\sigma_i^{-1}w_1\sigma_i$, where $\sigma_j^{\pm 1}$ is not in $w_1$ for $|i-j|\leq 1$. $w$ is $R_2$-reducible if it is not $R_2$-reduced. 
\end{enumerate}
We say that $w \in M_n$ is strongly-reduced if it is square free (positive and negative), nesting free, and $R_2$-reduced. Figure \ref{reductions} illustrates this definition.
\end{definition}

Observe that the set of positive strongly-reduced braid diagrams embeds into the set of strongly-reduced braid diagrams. Furthermore, if $w \in M_n$ is strongly-reduced, then so is $w^+\in M_n^+$.

\begin{lemma}(Positive square reduction Lemma)\label{SquareReduction}
Consider a braid diagram $w\in M_n$ that is not positive square free, and let the cyclic word $\tilde w \in \tilde M_n$
contains the subword $w_0=\sigma_iw_1\sigma_i$ where $w_1\in M_n$ does not contain a letter $\sigma_j$ with $|i-j|\leq 1$. Then
\begin{enumerate}
\item[(a)] The chord diagram $\varphi (\tilde w)$ has at least two components (circles) and one of them, say $C_1$ (see Figure \ref{reductions}(a)) has either Lando graph $G_{C_1} \subset G(w)$ with an isolated vertex and thus $I(G_{C_1})$ is contractible (this is the case when there is $u_i \in \varphi(w_1)$), or $G_{C_1}$ is empty, so $I(G_{C_1})\sim_h S^{-1}$ (this is the case if $u_i \notin \varphi(w_1)$).  

\item[(b)] Let $w_1'$ be the word in $M_n$ obtained from $w_1$ by deleting all letters of type $\sigma_j^{-1}$, $|i-j|\leq 1$. Define the new cyclic word $\tilde w'$ obtained form $\tilde w$ by replacing the subword $w_0= \sigma_iw_1\sigma_i$ by $w_0'=\sigma_iw'_1$. Then the Lando graph corresponding to the chord diagram $\varphi(\tilde w)- C_1$ is obtained from the Lando grap $G(\tilde w')$ by deleting the vertex $v_{\sigma_i}$ associated to $\sigma_i \in w_0'$.

\item[(c)] $I(w)$ is either contractible or it is the independence simplicial complex of an induced subgraph of a closed braid diagram which is positive square free. This process has polynomial time complexity.
\end{enumerate}
\end{lemma}

\begin{figure}
\centering
\includegraphics[width = 9cm]{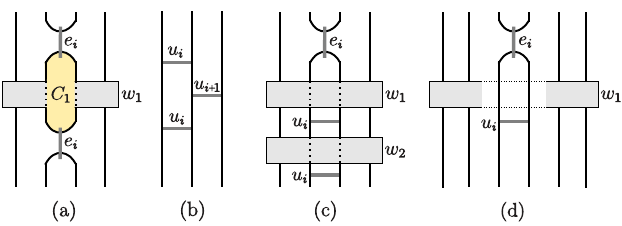}
\caption{\small{Situations forbidden in chord diagrams associated to positive square free (a), negative square free (b), nesting free (c) and $R_2$-reduced braid diagrams (d).}}
\label{reductions}
\end{figure}

\begin{proof} Parts (a) and (b) follow by careful analysis of chords attached to component $C_1$: the only essential chords, i.e. with both endpoints attached to the same circle, are those coming from $\sigma_i^{-1} \in w_1$. Part (c) follows from (a) and (b).
\end {proof}

\begin{lemma}\label{negativesquare} (Negative square reduction Lemma)
Let $w\in M_n$ and $w_0= \sigma_i^{-1}\sigma_{i\pm 1}^{-1}\sigma_i^{-1}$ be a subword of the cyclic word $\tilde w$. Then $I(w) \sim_h I(w')$, with $w'$ the braid diagram obtained from $w$ by deleting one of the $\sigma_i^{-1}$ occurrences in $w_0$.
\end{lemma}

\begin{proof} 
Write $v_1$ and $v_2$ for the vertices in $G(w)$ corresponding to each of the letters $\sigma_i^{-1}$ in $w_0$ (in case these vertices do not belong in $G(w)$, then the proof is trivial). It follows from Figure \ref{reductions}(b) that in $G(w)$ either $v_1 > v_2$ or $v_2 > v_1$. Therefore, Domination Lemma applies, allowing to eliminate the letter corresponding to the dominating vertex. 
\end{proof}

\begin{remark}
Lemma \ref{negativesquare} allows to apply classical braid relation on negative generators of $M_n$, i.e., a negative Reidemeister 3 move. In other words, if $w'$ is obtained from $w$ by replacing $\sigma_i^{-1}\sigma_{i\pm 1}^{-1}\sigma_i^{-1}$ by $\sigma_{i\pm 1}^{-1}\sigma_i^{-1}\sigma_{i\pm 1}^{-1}$, then $I(w) \sim_h I(w')$. Compare to last paragraph in Section \ref{Section 5}.
\end{remark}

\begin{lemma}\label{Nesting} Let $w\in M_n$ be a braid diagram and $w_0$ be a subword of the cyclic word $\tilde w$: 
\begin{enumerate}
 \item[(a)] (Nesting Lemma) Consider $w_0=\sigma_iw_1\sigma_i^{-1}w_2\sigma_i^{-1}$ where $\sigma_j$ is neither in $w_1$ nor in $w_2$, for $|i-j|\leq 1$. Then, $I(w) \sim I(w')$, with $\tilde{w}'$ the word obtained from $\tilde{w}$ by replacing $w_0$ by $w_0'=\sigma_iw_1\sigma_i^{-1}w_2$.  
 \item[(b)] Analogue of part (a) holds for the case when $w_0=\sigma_i^{-1} w_2 \sigma_i^{-1} w_1 \sigma_i$, with analogous conditions for $w_1$ and $w_2$.
 \item[(c)] ($R_2$-reduction Lemma) Consider  $w_0=\sigma_iw_1\sigma_i^{-1}$, where $\sigma_j^{\pm 1}$ is not in $w_1$ for $|i-j|\leq 1$. Then either $I(w)$ is contractible or $I(w) \sim_h \Sigma I(G(w) - st(v_{\sigma_i}))$, with $v_{\sigma_i}$ the vertex associated to $\sigma_i \in w_0$. In particular, $I(w) \sim_h \Sigma I(H)$, where $H$ is an induced subgraph of $G(w')$ and $\tilde{w}'$ is obtained from $\tilde{w}$ after replacing $w_0$ by $\sigma_iw_1$.
 \item[(d)] Analogue of part (c) holds for the case when $w_0=\sigma_i^{-1}w_1\sigma_i$ with analogous conditions for $w_1$.
\end{enumerate}
\end{lemma}

\begin{proof}
Part (a) follows from the fact that the vertex corresponding to the second occurrence of $\sigma_i^{-1}$ in $w_0$ dominates the vertex corresponding to the first occurrence of $\sigma_i^{-1}$, as shown in Figure \ref{reductions}(c); Domination Lemma completes the proof.

To prove part (c), see Figure \ref{reductions}(d) and notice that $\sigma_i^{-1} \in w_0$ gives rise to either an isolated vertex in $G(w)$, or to a leaf with preleaf $v_{\sigma_i}$, and Corollary~\ref{leaf} applies. 
\end{proof}

Lemmas \ref{SquareReduction}-\ref{Nesting} provide an algorithm reducing a braid diagram to a shorter braid diagram which is strongly-reduced. The exact meaning of this reduction is described in the next corollary which allows us to restrict ourselves to the study of strongly-reduced braid diagrams.

\begin{corollary}\label{Reduction}
 Given $w \in M_n$, there exists an algorithm of polynomial time complexity which either \begin{enumerate}
 \item[(a)] determines that $I(w)$ is contractible, or
 \item[(b)] finds a strongly-reduced braid diagram $w^{red}$ by deleting some letters from $w$, so that $I(w) \sim_h \Sigma^k I(H)$, with $H$ an induced subgraph of $G(w^{red})$ and $k$ a non-negative integer given by the algorithm.
\end{enumerate}
\end{corollary}

\section{Positive braid diagrams on 4 strands}\label{secpositive}

In this section we introduce a classification of strongly-reduced positive braid diagrams and use it to determine the homotopy type of the independence complexes associated to their closures. 

\begin{proposition}\label{B4classification}
Let $w\in M_4^+$ be a strongly-reduced positive braid diagram. Then, up to involution $\sigma_1 \leftrightarrow \sigma_3$, there exists a representative $\tilde{w}\in \tilde M_4$ of $w$ in one of the following families: 
\begin{itemize}
\item [$\mathcal{C}_0$:] $\{\varepsilon, \sigma_i$, \, $i=1,2,3\}$;
\item [$\mathcal{C}_{1}$:] $\{(\sigma_1\sigma_2)^{a}$, \,  $a \geq 1\}$;
\item [$\mathcal{C}_2$:] $\{(\sigma_1\sigma_2)^{a_1}(\sigma_3\sigma_2)^{a_2}\ldots(\sigma_3\sigma_2)^{a_{2k}}$, \, $a_i >0$, \, $k\geq 1\}$; 
\item [$\mathcal{C}_{3}$:] $\{\sigma_3(\sigma_1\sigma_2)^{a_1}(\sigma_3\sigma_2)^{a_2}\ldots(\sigma_3\sigma_2)^{a_{2k}}$, \, $a_i >0$, \, $k \geq 1\}$;
\item [$\mathcal{C}_{4}$:] $\{\sigma_3(\sigma_1\sigma_2)^{a_1}(\sigma_3\sigma_2)^{a_2}\ldots(\sigma_1\sigma_2)^{a_{2k+1}}$, \, $a_i >0$, \, $k \geq 0\}$;
\item [$\mathcal{C}_5$:] $\{\sigma_1\sigma_3 \, u \, \sigma_1\sigma_3 \, v$, \, $\sigma_1\sigma_3 \, u \, \sigma_3\sigma_1 \, v$, \, $u,v \in M_4^+\}$.
\end{itemize}
We say that $w$ (and $\tilde{w}$) belongs to the corresponding class $\mathcal{C}_i$ and write $w\in \mathcal{C}_i$, for $1\leq i\leq 5$. 
\end{proposition}

\begin{proof}
Since $w$ is strongly-reduced, $\tilde{w}$ contains no subwords $\sigma_1\sigma_3\sigma_1$. We start our classification from considering the number of $\sigma_1\sigma_3$ (or $\sigma_3 \sigma_1$) occurrences in $\tilde{w}$: if there is none of such occurrences, then we get families $\mathcal{C}_0, \mathcal{C}_1, \mathcal{C}_2$. Families $\mathcal{C}_3$ and $\mathcal{C}_4$ (resp. family $\mathcal{C}_5$) correspond to the cases when there is one (resp. at least two) of such occurrences.
\end{proof}

\begin{proposition}\label{c4toc3}
Given a braid diagram $w \in \mathcal{C}_4$ with $a_1>1$, there exists a braid diagram $w' \in \mathcal{C}_3$ so that their associated Lando graphs $G(w)$ and $G(w')$ are isomorphic. 
\end{proposition}

\begin{proof}
It is clear that involution $\sigma_1 \leftrightarrow \sigma_3$, cyclic permutation and reading a word backwards preserve Lando graph of a braid diagram. The following chain of transformations completes the proof: 
$$w= \sigma_3(\sigma_1\sigma_2)^{a_1}(\sigma_3\sigma_2)^{a_2}...(\sigma_1\sigma_2)^{a_{2k+1}}\stackrel{cyclic}{\longrightarrow} $$
$$(\sigma_2\sigma_1)^{a_1-1}(\sigma_2\sigma_3)^{a_2}...(\sigma_2\sigma_1)^{a_{2k+1}} (\sigma_2\sigma_3)\sigma_1\stackrel{backward}{\longrightarrow}$$
$$ \sigma_1(\sigma_3\sigma_2)(\sigma_1\sigma_2)^{a_{2k+1}}...(\sigma_1\sigma_2)^{a_1-1} \stackrel{\sigma_1 \leftrightarrow  \sigma_3}{\longrightarrow}$$
$$\sigma_3(\sigma_1\sigma_2)(\sigma_3\sigma_2)^{a_{2k+1}}...(\sigma_3\sigma_2)^{a_1-1} = w'.$$
\end{proof}

Observe that the statement in Proposition \ref{c4toc3} is true for any strongly-reduced braid diagram, not necessarily positive. More precisely, given a braid diagram $w$ such that $w^+ \in \mathcal{C}_4$ the proof above produces a braid $w'$ such that $w'^+ \in \mathcal{C}_3$. This will be useful in Section \ref{sectnegative}.

\begin{proposition}\label{Lando4pos}
Let $w \in M_4^+$ be a strongly-reduced positive braid diagram. Then:
\begin{itemize}
\item[(a)] If $w \in \mathcal{C}_0$, then $G(w)$ is either empty, if $w$ is the trivial word, or a single vertex otherwise. 
\item[(b)] If $w \in \mathcal{C}_1$, then $G(w)$ is ether a polygon of $2a$ edges if $a>1$, or two isolated vertices, if $a=1$. 
\item[(c)] If $w \in \mathcal{C}_2$, then $G(w)$ is a disjoint union of paths. Namely, $$G(w)= L_{2(a_1-1)} \sqcup L_{2(a_2-1)} \sqcup \ldots \sqcup L_{2(a_{2k}-1)}.$$
\item[(d)] If $w \in \mathcal{C}_3$, then $G(w)$ is a rhomboid graph (Definition~\ref{nosimpledef}). In particular, $$G(w) =
\left\{
  \begin{array}{lll}
     G(0,2(a_1-1),2a_2, \ldots,2a_{2k},0) & & \mbox{if } a_1>1, \\
G(2a_2,0,2a_3,\ldots,2a_{2k},0) & & \mbox{if } a_1=1.
  \end{array}
  \right.  $$
\item[(e)] If $w \in \mathcal{C}_4$, then $G(w)$ is a rhomboid graph. In particular, $$G(w) =
\left\{
  \begin{array}{lll}
     G(0,2(a_1-1),2a_2,\ldots,,2a_{2k},0,2a_{2k+1}) & & \mbox{if } a_1>1, \, k>0 \\
      G(2a_2,0,2a_3,\ldots,2a_{2k},0,2a_{2k+1}) & & \mbox{if } a_1=1, \, k>0,\\
     \{v\} \sqcup P_{2a_1} & & \mbox{if } k=0. 
  \end{array}
  \right.  $$
\item[(f)] If $w \in \mathcal{C}_5$, then $G(w) = G(D_1) \sqcup G(D_2)$, where $D_1$ and $D_2$ are $2$-bridge diagrams. 
\end{itemize}
\end{proposition}

\begin{proof}
Proof follows from direct analysis of the chord diagram $\varphi(\tilde{w})$ in each of the cases. 

Note that if $w \in \mathcal{C}_1$ then $\varphi(\tilde{w})$ corresponds to the disjoint union of the chord diagram of the torus link $T(3,a)$ and a circle, thus part (b) holds by \cite[Corollary 7.7]{JACO}. 

For part (c) observe that if $w \in \mathcal{C}_2$, then $\varphi(\tilde{w})$ consists of two circles. Each factor $(\sigma_i\sigma_2)^{a_j}$ gives rise to $2a_j$ chords (one of them non-admissible) yielding to a component $L_{2(a_j-1)}$ in $G(w)$, for $i=1,3$ and $1 \leq j \leq 2k$. 

Parts (d) and (e) are illustrated in Figures \ref{casec3a1big}-\ref{casec44a1big}. In these cases, $\varphi(\tilde{w})$ contains a single circle, vertices $b_r$ and $b_\ell$ correspond to chords associated to subword $\sigma_3\sigma_1$. Spoke vertices correspond to chords associated to the last letter of factors of the form $(\sigma_i\sigma_2)^{a_j}$ in $\tilde{w}$, for $i=1,3$. Also, if $a_1=1$ then the vertex corresponding to the $\sigma_2$ occurrence in $\sigma_3\sigma_1\sigma_2$ is connected to $b_r$ instead of $b_\ell$, so $c_0 = 2(a_1-1)$ becomes $c_0'=2a_2$, as shown in Figure \ref{casec3a1eq1}.

\begin{figure}
    \centering
\includegraphics[width = 12cm]{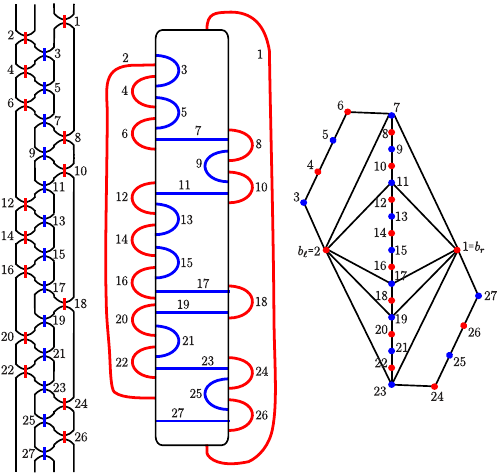}
\caption{\small{The chord diagram and Lando graph associated to $w= \sigma_3 (\sigma_1 \sigma_2)^3(\sigma_3 \sigma_2)^2(\sigma_1 \sigma_2)^3(\sigma_3 \sigma_2)(\sigma_1 \sigma_2)^2(\sigma_3 \sigma_2)^2$.}}\label{casec3a1big}
\end{figure}

\begin{figure}
    \centering
\includegraphics[width = 12cm]{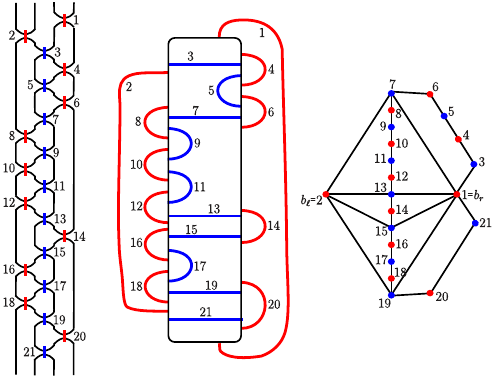}
\caption{\small{The chord diagram and Lando graph associated to $w= \sigma_3 (\sigma_1 \sigma_2)(\sigma_3 \sigma_2)^2(\sigma_1 \sigma_2)^3(\sigma_3 \sigma_2)(\sigma_1 \sigma_2)^2(\sigma_3 \sigma_2)$.}}\label{casec3a1eq1}
\end{figure}

\begin{figure}
    \centering
\includegraphics[width = 12cm]{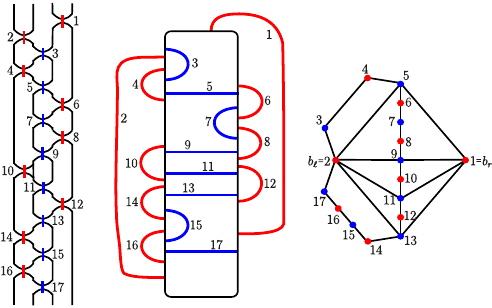}
\caption{\small{The chord diagram and Lando graph associated to $w= \sigma_3 (\sigma_1 \sigma_2)^2(\sigma_3 \sigma_2)^2(\sigma_1 \sigma_2)(\sigma_3 \sigma_2)(\sigma_1 \sigma_2)^2$.}}\label{casec44a1big}
\end{figure}

For part (f), observe that after performing a $B$-smoothing in the four crossings associated to two of the sets of occurrences $\sigma_1\sigma_3$ (or $\sigma_3\sigma_1$) in $\tilde{w}$, we obtain a disjoint sum of two $2$-bridge diagrams, together with four non-admissible chords.
\end{proof}

We finish this section with the proof of Corollary \ref{indepeasy}.

\begin{proof}[Proof of Corollary \ref{indepeasy}]
The first step is to apply Corollary~\ref{Reduction} to $w$ to obtain $w^{red}$. If $w^{red}$ belongs to families $\mathcal{C}_0 -  \mathcal{C}_4$ then Proposition \ref{Koz} and \ref{rhomboid2spheres} complete the proof. 

If $w^{red} \in \mathcal{C}_5$, we have $I(w) = I(G(D_1)) * I(G(D_2))$, with $D_1$ and $D_2$ two $2$-bridge diagrams (see Proposition \ref{Lando4pos}(f)). These diagrams represent alternating links, and therefore $I(G(D_i))$ is either contractible or homotopy equivalent to a sphere (see \cite{GeomDed}), for $i=1,2$.
\end{proof}

\section{Considering negative generators}\label{sectnegative}

In this section we assume all braid diagrams to be strongly-reduced, i.e., $w \in M_4^{red}$. We write $w \dot{\in} \mathcal{C}_i$ if $w \in M_4^{red}$ and $w^+\in \mathcal{C}_i$, for $0\leq i \leq 5$. 

\begin{proposition}\label{easycases}
Let $w\in M_4^{red}$. \begin{itemize}
\item[(a)] If $w \dot{\in} \mathcal{C}_0$, then $I(w)$ is homotopy equivalent to either $S^{-1} = \emptyset$ (if $w^+ = \varepsilon$) or $S^0$ (if $w^+= \sigma_i$ and $\sigma_i^{-1} \in w$), or it is contractible otherwise.
\item[(b)] If $w\dot{\in}  \mathcal{C}_1$, then $I(w)$ is either contractible or homotopy equivalent to either one sphere or the wedge of two spheres of the same dimension.
\item[(c)] If $w\dot{\in}  \mathcal{C}_2$, then $I(w)$ is either contractible or homotopy equivalent to a sphere.
\item[(d)] If $w\dot{\in} \mathcal{C}_5$, then $I(w)$ is either contractible or homotopy equivalent to a sphere. 
\end{itemize} 
Moreover, the complexity of computing the homotopy type of $I(w)$ is polynomial. 
\end{proposition}

\begin{proof}
Part $(a)$ is trivial. For part (b), recall that $G(w^+)$ corresponds to a $2a$-gon. $\sigma_1^{-1}$ and $\sigma_2^{-1}$ occurrences in $w$ yields to either a leaf or a square sharing a side with the $2a$-polygon corresponding to $w^+$ (see Figure \ref{braid3}). In both cases, the result holds by applying Domination Lemma together with Proposition \ref{Koz}. 

For part (c), notice that vertices arising from negative letters in $w$ do not connect the $2k$ paths in $G(w^+)$. As in the previous case, $\sigma_i^{-1}$ occurrences giving rise to admissible chords yields to  either leaves or squares sharing a side with one of the paths in $G(w^+)$, for $i=1,2,3$. Domination Lemma and Proposition \ref{Koz}(1) complete the proof.

For part (d) observe that negative $\sigma_2^{-1}$ occurrences between the two letters in $\sigma_1\sigma_3$ lead to non-admissible chords and therefore associated letters can be eliminated from $w$ to obtain a new braid diagram $w'$ with $I(w) = I(w')$. Thus the statement follows from Proposition \ref{Lando4pos} and the proof of Corollary~\ref{indepeasy}, since the diagrams $D_1$ and $D_2$ associated to $w'$ are 2-bridge diagrams, and therefore represent alternating links. 
\end{proof}

Computing the homotopy type of the independence complex of braid diagrams in families $\mathcal{C}_3$ and $\mathcal{C}_4$ is more involved and related to each other. We address these cases in the next subsections.

\subsection{The case $w \dot{\in} \mathcal{C}_3$}\label{sectionc3}

In this section we consider those braid diagrams $w$ on four strands which are strongly-reduced and such that $w^+$ belongs to the family $\mathcal{C}_3$ from Lemma \ref{B4classification}, that is, 
\begin{equation}\label{winc3}
\tilde{w}^+ = \sigma_3(\sigma_1 \sigma_2)^{a_1} (\sigma_3 \sigma_2)^{a_2} \cdots (\sigma_3\sigma_2)^{a_{2k}}
\end{equation} 
for some positive integers $a_i$ (with $1 \leq i \leq 2k$) and $k\geq 1$. Recall that we denote the above conditions as $w \dot{\in} \mathcal{C}_3$. 

\begin{definition}\label{defheadtail} Let $w \dot{\in} \mathcal{C}_3$. The head of $w$, denoted $w_H$, is the unique subword of $\tilde{w}$ in the form $w_H=\sigma_2 w_1 \sigma_3 w_2 \sigma_1 w_3 \sigma_2$, with $w_i \in M_4^-$, for $i=1,2,3$. The tail of $w$, $w_T$, is the subword of $\tilde{w}$ obtained by deleting the interior of the head $w_H$ (that is, following previous notation, we delete $w_1 \sigma_3 w_2 \sigma_1 w_3$).  
\end{definition}

\begin{proposition}\label{propoptions}
Let $w \dot{\in} \mathcal{C}_3$, with head $w_H = \sigma_2 w_1 \sigma_3 w_2 \sigma_1 w_3 \sigma_2$, where $w_i \in M_4^-$ for $i=1,2,3$. Then, strong reducibility of $w$ implies that the options for $w_2$ are $w_2= \emptyset$, $w_2= \sigma_2^{-1}$ or $w_2=\sigma_2^{-1}\sigma_1^{-1}\sigma_3^{-1}\sigma_2^{-1}$. Moreover: 
\begin{enumerate}
    \item[1)] If $w_2 = \emptyset$, then the options for $w_1$ and $w_3$ are: $$w_1= \emptyset, \quad w_1=\sigma_1^{-1}\sigma_2^{-1},  \quad w_1=\sigma_3^{-1}\sigma_2^{-1}, \quad    w_1=\sigma_1^{-1}\sigma_3^{-1}\sigma_2^{-1},$$ $$w_3= \emptyset, \quad  w_3=\sigma_2^{-1}\sigma_3^{-1}, \quad w_3=\sigma_2^{-1}\sigma_1^{-1}, \quad  w_3=\sigma_2^{-1}\sigma_1^{-1}\sigma_3^{-1}.$$ 
    
\item[2)] If $w_2= \sigma_2^{-1}$, then we have to consider, in addition to previous options for $w_1$ and $w_3$, the following: $$w_1=\sigma_1^{-1}, \quad w_1=\sigma_3^{-1}\sigma_2^{-1}\sigma_1^{-1}, \quad  w_3=\sigma_3^{-1}, \quad  w_3=\sigma_3^{-1}\sigma_2^{-1}\sigma_1^{-1}.$$

\item[3)] If $w_2=\sigma_2^{-1}\sigma_1^{-1}\sigma_3^{-1}\sigma_2^{-1}$, then the options for $w_1$ and $w_3$ are:
$$w_1=\emptyset, \quad w_1=\sigma_3^{-1}\sigma_2^{-1}, \quad w_3=\emptyset, \quad w_3=\sigma_2^{-1}\sigma_1^{-1}.$$
\end{enumerate}
\end{proposition}

\begin{proof}
Given a word $w$, we write $w(j)$ for the $j^{th}$ letter in $w$ and denote by $n_i$ the length of $w_i$, that is, $w_i(n_i)$ is the last letter of $w_i$, for $i= 1,2,3$.

\begin{figure}
    \centering
\includegraphics[width = 11cm]{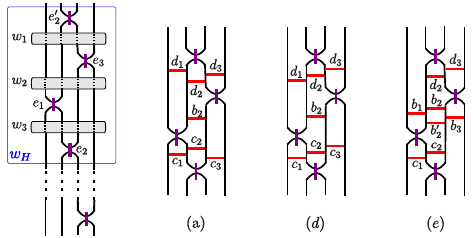}
\caption{\small{A scheme of the chord diagram of a braid $w \dot{\in} \mathcal{C}_3$ illustrating proof of Proposition \ref{propoptions}. The cases (a), (d) and (e) for maximal heads $w_H$ described after Remark \ref{remarkinvo} are shown.}} 
\label{head1}
\end{figure}

See Figure \ref{head1}. Since $w$ is strongly-reduced (and therefore nesting free and $R_2$-reduced) none of $w_i$ with $i=1,2,3$ contains two repeated letters, with the unique possible exception of $w_2=\sigma_2^{-1}\sigma_1^{-1}\sigma_3^{-1}\sigma_2^{-1}$. Furthermore, $w_2(1) \neq \sigma_3^{-1}$ and $w_2(n_2) \neq \sigma_1^{-1}$. Moreover, if $w_2(1) = \sigma_1^{-1}$ then we can redefine $w_1' = w_1 \cdot \sigma_1^{-1}$ and $w_2'=w_2(2) \cdots w_2(n_2)$, since replacing $\sigma_1^{-1}\sigma_3$ by $\sigma_3\sigma_1^{-1}$ preserves Lando graph; analogous reasoning works when $w_3(n_3)=\sigma_3^{-1}$. Hence, $w_2$ is either empty or $w_2(1) = w_2(n_2)= \sigma_2^{-1}$.

Strong reducibility of $w$ also implies that $w_1(1) \neq \sigma_2^{-1} \neq w_3(n_3)$, \, $w_1(n_1) \neq \sigma_3^{-1}$ and $w_3(1)\neq \sigma_1^{-1}$. If $w_2= \emptyset$ we have two additional restrictions $w_1(n_1) \neq \sigma_1^{-1}$ and $w_3(1) \neq \sigma_3^{-1}$. This completes the proof of (1) and (2).\\
For case (3) there are two additional restrictions coming from nesting free property of $w$: $\sigma_1^{-1}\notin w_1$ and $\sigma_3 \notin w_3$. 
\end{proof}

\begin{remark}\label{remarkinvo}
It is natural, by symmetry in Figure \ref{head1}, that each option for $w_1$ in Proposition \ref{propoptions} is associated by reversing order and involution $\sigma_1^{-1} \leftrightarrow \sigma_3^{-1}$ to one of the options for $w_3$.
\end{remark}

For simplicity, from now on we assume that $a_1>1$ when expressing $\tilde{w}^+$ as in relation \eqref{winc3}. When $a_1=1$ the Lando graphs to analyze are simpler (compare Figures \ref{casec3a1big} and \ref{casec3a1eq1}). Thus, we get similar results (with analogous proofs) for the case $a_1=1$.

In Proposition \ref{Lando4pos} we determined the Lando graph associated to positive (strongly-reduced) braid diagrams; in the particular case when $w^+\in \mathcal{C}_3$ we determined that $G(w^+)$ is a rhomboid graph. Now, we will study how such a rhomboid graph is modified by adding vertices (and edges) corresponding to negative letters in $w \dot{\in} \mathcal{C}_3$.

Notice that all possible combinations for the head $w_H$ listed in Proposition~\ref{propoptions} are subwords (not necessarily with consecutive letters) of one of the combinations listed below. We refer to them as {\it{maximal heads}} (see Figure~\ref{head1}). \\
(a) $w_1= \sigma_1^{-1} \sigma_3^{-1}\sigma_2^{-1}$, \quad  $w_2=\sigma_2^{-1}$, \quad  $w_3= \sigma_2^{-1} \sigma_1^{-1}\sigma_3^{-1}$. \\
(b) $w_1=\sigma_1^{-1} \sigma_3^{-1}\sigma_2^{-1}$, \quad $w_2=\sigma_2^{-1}$, \quad $w_3=\sigma_3^{-1} \sigma_2^{-1} \sigma_1^{-1}$. \\
(c) $w_1=\sigma_3^{-1} \sigma_2^{-1} \sigma_1^{-1}$, \quad $w_2=\sigma_2^{-1}$, \quad $w_3=\sigma_2^{-1} \sigma_1^{-1}\sigma_3^{-1}$. \\
(d) $w_1=\sigma_3^{-1} \sigma_2^{-1} \sigma_1^{-1}$, \quad $w_2=\sigma_2^{-1}$, \quad $w_3=\sigma_3^{-1} \sigma_2^{-1} \sigma_1^{-1}$. \\
(e) $w_1= \sigma_3^{-1} \sigma_2^{-1}$, \quad $w_2 = \sigma_2^{-1} \sigma_1^{-1}\sigma_3^{-1} \sigma_2^{-1}$, \quad $w_3=\sigma_2^{-1} \sigma_1^{-1}$. \\

Figure \ref{romboidshead} illustrates the Lando graphs of braid diagrams with maximal head and positive tail. Observe the graph (i) corresponding to maximal head (a) is an augmented rhomboid graph (Definition \ref{nosimpledef}). 

\begin{figure}
    \centering
\includegraphics[width = 11cm]{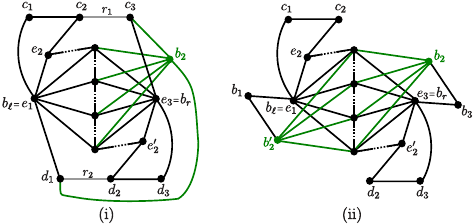}
\caption{\small{Lando graphs associated to braid diagrams with maximal heads in the case when $a_1>1$. Graph (i) corresponds to case (a). Cases (b) and (c) are obtained from graph (i) after removing edge $r_1$ and $r_2$, respectively. Removing $r_1$ and $r_2$ from $(i)$ corresponds to case (d). Graph (ii) corresponds to case (e).}} 
\label{romboidshead}
\end{figure}

\begin{proposition}\label{subwordofmaximalhead}
Let $w \dot{\in}\mathcal{C}_3$ with positive tail $w_T$. There exists a braid diagram $w' \dot{\in}\mathcal{C}_3$ satisfying that $w'_H$ is a maximal head extending $w_H$, $w_T = w_T'$ and $G(w)$ is an induced subgraph of $G(w')$ (i.e, an induced subgraph of one of the graphs depicted in Figure \ref{romboidshead}). Moreover, the homotopy type of $I(w)$ can be computed in polynomial time. 
\end{proposition}

The first part of Proposition \ref{subwordofmaximalhead} follows from discussion above. The second part follows from the following propositions:

\begin{proposition}\label{propfig17b}
Let $G$ be an induced subgraph of the graph in Figure~\ref{romboidshead}(ii). Then $I(G)$ is either contractible or has the homotopy type of either a sphere or the wedge of two spheres, and this can be determined in polynomial time. 
\end{proposition}

\begin{proof}
We assume that $G$ contains vertices $e_1, e_3, b_2$ and $b_2'$ (otherwise the result follows easily). It is clear that $e_3$ and $e_1$ dominates $b_2$ and $b_2'$ respectively. Therefore, $I(w) \sim_h I(G(w)-e_1 -e_3)$. The graph $G(w) - e_1 - e_3$ is an induced subgraph of a rhomboid graph (with $b_2$ and $b_2'$ taking the role of vertices $b_r$ and $b_\ell$, respectively), and therefore Proposition \ref{rhomboid2spheres} completes the proof. 
\end{proof}

\begin{proposition}\label{propfig17a}
Let $G'$ be one of the graphs illustrated in Figure~\ref{romboidshead}(i), that is, the graph in black together with any subset of $\{r_1, r_2\}$, and let $G$ be an induced subgraph of $G'$. Then, $I(G)\sim_h \Sigma^k I(H)$, where $H$ is an induced subgraph of a mod-augmented simple rhomboid graph (Definition~\ref{modaugmented}). Moreover the algorithm to compute $H$ and $k$ from $G$ has polynomial complexity, and therefore the homotopy type of $I(G)$ can be computed in polynomial time. 
\end{proposition}

\begin{proof}
The algorithm to find $H$ uses Csorba reduction (Proposition~\ref{genCsorba}), Domination Lemma (Lemma~\ref{domlemma}) and application of Corollary \ref{leaf} whenever we get a graph containing a leaf. 

First step is applying Csorba reduction in the spine of $G$. If the distance between two consecutive spoke vertices $v_i$ and $v_{i+1}$ is $2$, then the vertex between them is dominated by $b_\ell$, $b_r$ and $b_2$ (in case they are in $G$). Therefore $I(G) \sim_h I(G-b_\ell-b_r-b_2)$, and this graph is a forest and (after possibly applying Csorba reductions) can be seen as an induced subgraph of an augmented simple rhomboid graph.

Thus, assume that the distance between two consecutive spoke vertices $v_i$ and $v_{i+1}$ is either $0, 1$ or $\infty$, in case they are not connected (recall Proposition~\ref{trickremovingedge}), for $1\leq i \leq n-1$. We keep the name $G$ for the graph obtained after Csorba reductions in the spine.

We focus now in the {\it{upper part}} of $G$ (that is, the part involving vertices $c_1,c_2,c_3,e_2,v_0$, in case they belong to $G$, and the vertices adjacent to them). The goal is to transform it into the upper part of an induced subgraph of a mod-augmented simple rhomboid graph (compare to Figure \ref{defrhomboids}(b), illustrating an augmented simple rhomboid graph). We assume without loss of generality that vertex $e_2$ and edge $r_1$ belong to $G$ and $c_1 \notin G$ (otherwise, apply Corollary \ref{leaf} and Domination Lemma over $e_2>c_1$ if needed). 

Next step is to apply Csorba reduction in the path connecting $e_2$ and the first spoke vertex $v_0$; we write $c_0$ for its distance: \\
$\cdot$ If $c_0 = 0$ $(mod \, 3)$, then after Csorba reduction vertices $v_0$ and $e_2$ are identified (this corresponds to transformation (a) in Definition \ref{modaugmented}). \\
$\cdot$ If $c_0 = 2$ $(mod \, 3)$, write $v$ for the vertex between $e_2$ and $v_0$. It is clear that vertex $b_\ell$ dominates $v$ (if $b_\ell$ belongs to $G$), so $I(G) \sim_h I(G-b_\ell)$. Then the upper part of $G-b_\ell$ can be seen (possibly after applying Csorba reduction on $c_2-e_2-v-v_0$) as the upper part of an induced subgraph of a simple rhomboid graph, where $b_2$ and $c_3$ take the role of $b_\ell$ and $v_0$, respectively. \\
$\cdot$ If $c_0 = \infty$, then the upper part of $G$ coincides with the upper part of the induced subgraph of an augmented simple rhomboid graph where the edge connecting $v_{-1}$ and $v_0$ has been removed (transformation (c) in Definition~\ref{modaugmented}). \\
$\cdot$ If $c_0 = 1$ $(mod \, 3)$, then just apply Csorba reduction a finite number of times until $c_0=1$. 

Symmetric analysis holds for the {\it{lower part}} of the graph, leading to transformations (b) and (d) in Definition \ref{modaugmented}. Corollary \ref{cormodaugmented4spheres} completes the proof. 
\end{proof}

\begin{figure}
    \centering
\includegraphics[width = 11cm]{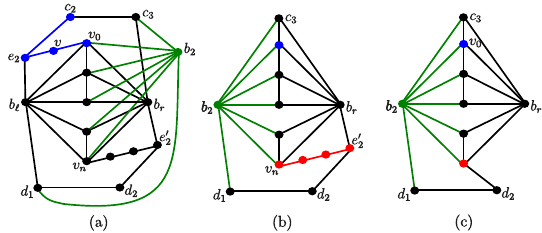}
\caption{\small{The graphs illustrating Example~\ref{ejfinal}.}} 
\label{exfinal}
\end{figure}

\begin{example}\label{ejfinal}
Let $G$ be the graph depicted in Figure \ref{exfinal}(a), which is an induced subgraph of one of the graphs depicted in Figure \ref{romboidshead}(i). Vertex $b_\ell$ dominates $v$ so Domination Lemma applies, and after applying Csorba reduction over the path $c_2-e_2-v-v_0$, we get that $I(G)\sim_h \Sigma I(G_1)$, with $G_1$ the graph in Figure \ref{exfinal}(b). Applying once more Csorba reduction over the path conecting $v_n$ and $e_2'$ in $G_1$ we obtain $I(G_1) \sim_h \Sigma I(H)$, with $H$ the graph in Figure \ref{exfinal}(c), which is an induced subgraph of a mod-agumented simple rhomboid graph. 
\end{example}

The following result, which follows from analysis of Figure~\ref{romboidshead}, will be useful when considering negative letters in the tail of a braid diagram, which we address in the next section.

\begin{lemma}\label{claim} Let $w \dot{\in}\mathcal{C}_3$ with maximal head $w_H$ and positive tail $w_T$, and let $a \in w_H$ be a negative letter and $v_a$ its associated vertex in $G(w)$. Then:
\begin{itemize}
\item[(1)] If $a= \sigma_2^{-1} \in w_2$, then $v_a$ is adjacent to all spoke vertices. 
\item[(2)] Otherwise, $v_a$ is not adjacent to any vertex in $P-e_2-e_2'$, where $P$ is the extended spine (see Definition \ref{defallrhomboids}) of $G(w)$.
\end{itemize}
\end{lemma}

\subsubsection{Elimination of negative letters in tail $w_T$}\label{secttail}

Proposition \ref{subwordofmaximalhead} and Corollary \ref{cormodaugmented4spheres} provide a proof of Theorem \ref{theomain} for those braid words having positive tail. In this section we show that, given a braid diagram $w\dot{\in}\mathcal{C}_3$, it is possible to reduce (in polynomial time) all negative letters from $w_T$ to obtain a braid $w'\dot{\in}\mathcal{C}_3$ with no negative letters in its tail $w'_T$ and determine how $I(w')$ is related to $I(w)$. We first eliminate those $\sigma_2^{-1}$ occurrences in $w_T$, and then letters $\sigma_1^{-1}$ and $\sigma_3^{-1}$.

\begin{lemma}\label{elimination2tail} (Elimination of $\sigma_2^{-1}$ in $w_T$)
Let $w \in M_4^{red}$ and let $w_0$ be a subword  of $\tilde w$ of the form $w_0= \sigma_2w_1\sigma_1w_2\sigma_2$, with $w_i\in M^-_4$ and $w_0$ containing the letter $\sigma_2^{-1}$. Replace $w_0$ by $w_0' = \sigma_2w_1'\sigma_1\sigma_2\sigma_1w_2'\sigma_2$ in $\tilde{w}$ to get $\tilde{w}'$, where $w_i'$ is obtained from $w_i$ by deleting $\sigma_2^{-1}$ occurrences, for $i=1,2$. Write $v_1, v_2, v_1'$ for the vertices of $G(w')$ associated to each of the letters of the central subword $\sigma_1\sigma_2\sigma_1$. Then
\begin{enumerate}
\item $I(w) \sim_h I(G(w')-v_2)$, \quad \quad \quad if $w_1$ and $w_2$ contain $\sigma_2^{-1}$,
\item $I(w) \sim_h I(G(w')-\{v_1,v_2\})$ \quad if $\sigma_2^{-1} \notin w_1$,
\item $I(w) \sim_h I(G(w')-\{v_1',v_2\})$ \quad if $\sigma_2^{-1} \notin w_2$.
\end{enumerate}
\end{lemma}

\begin{proof} First assume both $w_1$ and $w_2$ contain $\sigma_2^{-1}$. See Figure \ref{splittingvertex}(a). It is clear that the vertex $e_1$ associated to the unique occurrence $\sigma_1\in w_0$ dominates both vertices $u_2$ and $u_2'$ associated to $\sigma_2^{-1}$ in $w_1$ and $w_2$. Then, Domination Lemma applies and $I(w) \sim_h I(G(w)-e_1)$. However, the graph $G(w)-e_1$ is isomorphic to the graph $G(w')-v_2$, as shown in Figure \ref{splittingvertex}(b); vertices $u_2$ and $u_2'$ in $G(w)-e_1$ take the role of $v_1$ and $v_1'$ in $G(w')-v_2$. This completes the proof of part (1). The same reasoning works for parts (2) and (3), removing from Figure \ref{splittingvertex} the vertices that are not in the diagram. 
\end{proof}

\begin{figure}
    \centering
\includegraphics[width = 13cm]{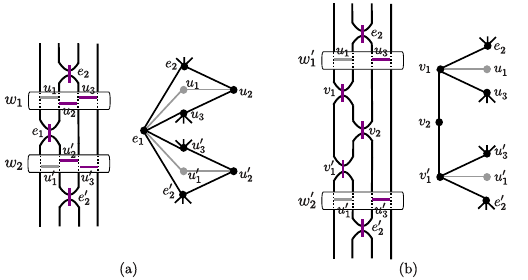}
\caption{\small{The chord diagram associated to $w$ (resp. $w'$) and the graphs $G(w)$ (resp. $G(w')$) illustrating proof of Lemma \ref{elimination2tail} are shown in (a) (resp. (b)).}}
\label{splittingvertex}
\end{figure}

\begin{remark}
By symmetry, it is straightforward that Lemma \ref{elimination2tail} holds when $w_0=\sigma_2w_1\sigma_3w_2\sigma_2$, with $w_1, w_2\in M_4^{-}$. 
\end{remark}

We call the idea in proof of Lemma \ref{elimination2tail} the {\it splitting vertex method}. It allows to transform a braid diagram $w \dot{\in} \mathcal{C}_3$, into another braid diagram $w'\dot{\in}  \mathcal{C}_3$ with $\sigma_2^{-1} \notin w'_T$ satisfying $I(w) \sim_h I(H)$, with $H$ an induced subgraph of $G(w')$. Hence, from now on we assume $w \dot{\in} \mathcal{C}_3$ not to contain letters $\sigma_2^{-1}$  in $w_T$. Next, we show how to eliminate $\sigma_1^{-1}$ and $\sigma_3^{-1}$ from $w_T$. 

\begin{lemma}\label{lemaspider}
Let $w\dot\in\mathcal{C}_3$ with $\sigma_2^{-1} \notin w_T$. Then, the vertices associated to each occurrence $\sigma_1^{-1}$ and $\sigma_3^{-1}$ in $w_T$ are adjacent to at most five vertices in $G(w)$. 
\end{lemma}

\begin{figure}
    \centering
\includegraphics[width = 11cm]{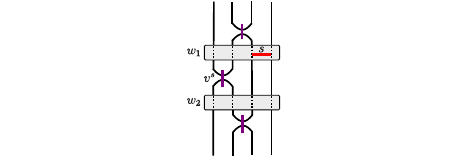}
\caption{\small{The chord diagram corresponding to the subword $w_0=\sigma_2 w_1 \sigma_1 w_2 \sigma_2 \in w_T$, with $w_1,w_2\in M_4^-$, illustrating the proof of Lemma \ref{lemaspider}.}}
\label{spidercase}
\end{figure}

\begin{proof}
Let $\sigma_3^{-1} \in w_T$, and denote the associated vertex in $G(w)$ by $s$. Since $w$ is strongly-reduced, then $\sigma_3^{-1}$ belongs to a subword $w_0=\sigma_2 w_1 \sigma_1 w_2 \sigma_2$ of $\tilde{w}$, with $w_1,w_2\in M_4^-$. Assume, without loss of generality, that $\sigma_3^{-1}\in w_1$; then, by strong reducibility (nesting free and $R_2$-reduced conditions), $w_1 = \sigma_3^{-1}$ and $w_2$ is trivial. It follows from Figures \ref{casec3a1big} and \ref{spidercase} that the vertex corresponding to $\sigma_1$ in $w_0$ is the unique vertex from $w_T$ which is adjacent to $s$ in $G(w)$. We denote it by $v^s$.

Next, we study which vertices from the head $w_H =\sigma_2 w_1^H \sigma_3 w_2^H \sigma_1w_3^H \sigma_2$, with $w_i^H\in M_4^-$ for $i=1,2,3$, are adjacent to $s$. It is clear from Figure \ref{head1} that the only vertices of $w_H$ adjacent to $s$ in $G(w)$ are those corresponding to $\sigma_1$ and $\sigma_3$ together with occurrences of $\sigma_2^{-1}$ in $w_2^H$. However, since $w$ is strongly-reduced it follows from  Proposition~\ref{propoptions} that $w_2^H$ contains at most two $\sigma_2^{-1}$ occurrences.
\end{proof}

\begin{definition}\label{defspider}
Let $w \dot{\in}\mathcal{C}_3$ with $\sigma_2^{-1}\notin w_T$. The vertices associated to letters $\sigma_1^{-1}$ and $\sigma_3^{-1}$ in $w_T$ are called spiders. Each spider $s$ is adjacent to at most five vertices:\\
- vertex $b_\ell$ associated to $\sigma_1 \in w_H$. \\
- vertex $b_r$ associated to $\sigma_3 \in w_H$. \\
- vertices $b_2$, $b_2'$ associated to occurrences $\sigma_2^{-1} \in w_2^H \subset w_H$. \\
- vertex $v^s$, associated to the unique vertex from $w_T$ adjacent to $s$ (for details see proof of Lemma \ref{lemaspider}). If we write $P$ for the extended spine of $G(w)$, then $v^s$ is a vertex in $P - e_2 - e_2'$. 
\end{definition}

\begin{remark}\label{remspi}
Observe (see Figure \ref{romboidshead}) that vertices $b_\ell, b_r$, $b_2$ and $b_2'$ are adjacent to all spoke vertices in $G(w)$, and therefore if we consider an induced subgraph not containing the vertex $v^s$ associated to a spider $s$, then all spoke vertices and other spiders dominate $s$ and we apply Domination Lemma and further reductions in Section \ref{secbasic} to get the complete graph $K_4$ (with $I(K_4) \sim_h S^0 \vee S^0 \vee S^0$) or an induced subgraph of $K_4$. Therefore, $I(G)$ is either contractible or has the homotopy type of the wedge of at most 3 spheres of the same dimension.
\end{remark}

We are now ready to present the next step of our algorithm, which allows to eliminate spiders (i.e., $\sigma_1^{-1}$ and $\sigma_3^{-1}$ occurrences) from the tail of a given braid diagram. 

Let $w \dot{\in} \mathcal{C}_3$ and assume $\sigma_2^{-1}\notin w_T$ and $G=G(w)$ contains no leaves (otherwise, apply Corollary \ref{leaf}). Given a spider $s$, define $d^s$ as the shortest distance between $v^s$ and a spoke vertex $\{v_i\}_{i=0}^n$ taken along the extended spine  $P$ of $G$; we set $d^s = \infty$ if $v^s$ is not connected to any spoke vertex along $P$. Let $d = \min \{d^s \, | s \, \, \mbox{is a spider in } G\}$. We eliminate spiders recursively; in each step we eliminate those spiders $s$ satisfying $d^s=d$. Observe that Proposition \ref{genCsorba} and Lemma \ref{claim} imply that we can restrict to the case when $d \in\{0,1,2,\infty\}$. Next we analyze each of these cases. \\

\begin{enumerate}

\item[(1)] \textbf{Case $d=0$}.
Vertex $v^s$ is a spoke vertex and therefore it dominates $s$, so by Domination Lemma $I(w) \sim_H I(G-v^s) \vee \Sigma I(G-st(v^s)).$

By Remark \ref{remspi}, $I(G-v^s)$ is either contractible or has the homotopy type of a sphere or two spheres or three spheres of the same dimension. Moreover, in the graph $G-st(v^s)$ all spiders become leaves and they can be removed by using Corollary~\ref{leaf}, leading to a forest; thus, $I(G-st(v^s))$ is either contractible or homotopy equivalent to a sphere. The fact that the dimensions of the spheres arising in $I(G-v^s)$ are smaller or equal than the dimension of the sphere in $I(G-st(v))$ follows from Lemma \ref{Lemma123}. In particular, Theorem \ref{theomain} holds for $w$.\\

\item[(2)] \textbf{Case $d=1$}.
See Figure \ref{cased1}(a). Spoke vertex $v_i$ dominates $s$, so $I(G) \sim_h I(G-v_i)$. Graph $G-v_i$ can be thought as removing $s$ from $G$ together with edge $e$; applying Proposition \ref{trickremovingedge} completes this case, allowing to remove vertex $s$. \\ 

\begin{figure}
    \centering
\includegraphics[width = 11.5cm]{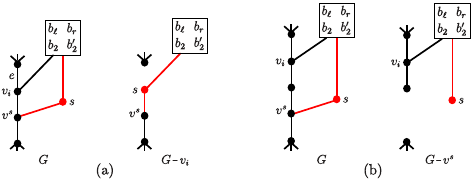}
\caption{\small{Graphs in the cases when the vertex $v^s$ associated to a spider $s$ is at distance 1 and 2 from spoke vertex $v_i$ are illustrated in (a) and (b), respectively. An edge connected to the box represents four edges connected to each of the vertices in the box, in case these vertices are in $G$.}}
\label{cased1}
\end{figure}

\item[(3)] \textbf{Case $d=2$}.
See Figure \ref{cased1}(b). Spoke vertex $v_i$ dominates $s$ in $G-v^s$, so $I(G-v^s) \sim_h I(G-v^s-v_i)$, which is contractible. Therefore, Proposition \ref{Prop3.3} applies and $$I(G) \sim_h I(G-v^s) \vee \Sigma I(G-st(v^s)) \sim_h \Sigma I(G-st(v^s)),$$ and therefore $s$ can be removed. \\

\item[(4)] \textbf{Case $d=\infty$}.
In this case, there exist no path along the extended spine $P$ connecting $v^s$ to any spoke vertex, for every spider $s$. Since $G$ has no leaves (otherwise, we apply Corollary \ref{leaf}) 
Proposition \ref{genCsorba} implies that there are essentially three possibilities: \\

\noindent(i) There exist two spiders $s_1$ and $s_2$ so that $v^{s_1}$ and $v^{s_2}$ are connected by an edge. See Figure \ref{casedinfinito}(a). In $G-v^{s_2}$ the vertex $v^{s_1}$ becomes a leaf and therefore $I(G-v^{s_2}) \sim_h \Sigma I(G-v^{s_2}-st(s_1))$, which is contractible. Thus Proposition \ref{Prop3.3} leads to $I(G) \sim_h \Sigma I(G-st(v^{s_2}))$, which is either contractible or has the homotopy type of the wedge of at most 3 spheres of the same dimension, by Remark \ref{remspi}. \\

\begin{figure}
    \centering
\includegraphics[width = 12cm]{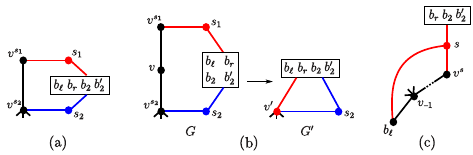}\caption{\small{Graphs illustrating the graphs involved in the case when $d= \infty$. Cases (a), (b) and (c) correspond to situations (i), (ii) and (iii), respectively. As before, an edge connected to a box represents edges connected to each of the vertices in the box, in case these vertices are in the graph.}}
\label{casedinfinito}
\end{figure}

\noindent(ii) There exist two spiders $s_1$ and $s_2$ so that $v^{s_1}$ and $v^{s_2}$ are at distance 2. See Figure \ref{casedinfinito}(b). Consider the graph $G'$ obtained from $G$ after collapsing the path of length 3 connecting $s_1$ and $v^{s_2}$, and write $v'$ for the resulting vertex. By Proposition \ref{genCsorba} we get that $I(G) \sim_h \Sigma I(G')$, and since $v'$ dominates $s_2$ in $G'$, Domination Lemma implies that $$I(G') \sim_h I(G'-v')\vee \Sigma I(G'-st(v')).$$ 
In $G'-v'$ we apply Remark \ref{remspi}, while the graph $G'-st(v')$ is a forest and then Proposition \ref{Koz}(2) applies. Lemma \ref{Lemma123} implies that the dimensions of the spheres in $I(G'-v')$ are smaller or equal than that of the sphere arising in $I(G'-st(v'))$. \\

\noindent (iii) There is a spider $s$ connected by a path to vertex $v_{-1}$ as shown in Figure \ref{casedinfinito}(c). In this case we can think of $G$ as the induced graph of an augmented rhomboid graph $G'$ where spoke vertices are defined in terms of vertices of $G$ as $v_0'=s$, $v_{i+1}'=v_i$, for $0 \leq i \leq n+1$. \\
\end{enumerate}

Repeated application of the methods described in this section (Lemma~\ref{elimination2tail} together with discussion of cases depending on value of $d$) gives rise to the following result:

\begin{proposition}\label{reducetopositivetail}
Given a braid diagram $w \dot{\in} \mathcal{C}_3$, there exists a polynomial complexity algorithm which either determines directly the homotopy type of $I(G)$ and it satisfies conditions in Theorem \ref{theomain}, or it produces a braid diagram $w'\dot{\in} \mathcal{C}_3$ with no negative letters in $w_T'$ with the property that $I(w)$ is obtained from $I(G')$, where $G'$ is an induced subgraph of $G(w')$. The precise way to obtain $I(w)$ from $I(G')$ is given by the algorithm and it consists of taking a finite number of suspensions and, possibly, a wedge operation.
\end{proposition}

Proposition \ref{reducetopositivetail} together with Proposition \ref{subwordofmaximalhead} complete the proof of the main theorem in the case when the given braid diagram belongs to the family $\mathcal{C}_3$ and $a_1>1$. The remaining cases when $w \dot{\in} \mathcal{C}_3$ with $a_1=1$ or $w \dot{\in} \mathcal{C}_4$ follow in a similar manner (in fact, the case when $w \dot{\in} \mathcal{C}_4$ and $a_1>1$ follows from Proposition \ref{c4toc3}). This completes the proof of Theorem \ref{theomain}. For convenience of the reader we finish this section with a sketch of the algorithm which, given a $4$-braid diagram as input, determines the homotopy type of $I(w)$ in polynomial time. 

\vspace{0.55cm}

\hrule

\vspace{0.22cm}

\noindent \textbf{Input}: a $4$-braid diagram $w\in M_4$ \vspace{0.2cm}

$\bullet$ Find $w_1:=w^{red}$ strongly-reduced braid diagram 
(Corollary~\ref{Reduction}) \vspace{0.2cm}

$\bullet$  Identify $j$ so that $w_1\dot{\in}\mathcal{C}_j$, $0\leq j \leq 5$ (Proposition~\ref{B4classification}) \vspace{0.2cm}

\hspace{0.4cm} $\cdot$ If $j=0,1,2,5$, then apply Proposition \ref{easycases} \vspace{0.2cm}

\hspace{0.4cm} $\cdot$ Otherwise, use Proposition \ref{reducetopositivetail} to eliminate negative letters in $(w_1)_T$ \mbox{ } \hspace{0.7cm} and apply Proposition \ref{subwordofmaximalhead} \vspace{0.2cm} 

\noindent \textbf{Return} homotopy type of $I(w)$

\vspace{0.22cm}

\hrule

\vspace{0.1cm}

\section{Concluding remarks and future directions}\label{secfinal}
We believe that our main computational complexity conjecture (Conjecture \ref{Conjecture 1}) will be the central object of research for years to come.

The small steps we plan to address in the future is to work on almost extreme grades of Khovanov homology; compare to \cite{FedSil2,GeomDed}. The case of positive $5$-braids and extreme Khovanov homology seems to be possible to approach today. Generally, some new techniques must be discovered for a general proof of Conjecture \ref{Conjecture 1}. There have been several unsuccessful attempts to prove the wedge of spheres conjecture (Conjecture~\ref{Conjecture 3}), which seems to be rather difficult but very attractive. 

\subsection{Khovanov adequate diagrams}\label{sectionkhovad}
We conjecture that links $L$ whose Khovanov homology $Kh_{*,j_{\min}}(L)$ is different from every group in the statement of Corollary \ref{summarykhov} have braid index greater than four. Our results prove it only partially as in Corollary \ref{Adeq}. Before its statement we introduce a formal definition of {\it Khovanov adequate} diagrams and links.

\begin{definition}\label{KhovanovAdequate}A diagram $D$ is Khovanov $B$-adequate if its Khovanov homology in quantum grading $j_{\min}(D)=-c-2|s_B|$ is not trivial. Analogously, $D$ is Khovanov $A$-adequate if its Khovanov homology in quantum grading \, $j_{\max}:=\max\{j(s,\epsilon) \, | \, (s, \epsilon) \mbox{ is an enhanced state of } D \}$  is not trivial. A diagram satisfying both properties is said to be Khovanov adequate. 

A link is Khovanov $B$-adequate (resp. $A$-adequate or adequate) if there exists a Khovanov $B$-adequate (resp. $A$-adequate or adequate) diagram representing it.
\end{definition}

As a consequence of Theorem \ref{theomain} we get the following result. 

\begin{corollary}\label{Adeq} Let $L$ be a link so that the Khovanov homology of the lowest non-trivial row (i.e., the one with minimal quantum grading) is different from every group in the statement of Corollary \ref{summarykhov}. Then there is no Khovanov $B$-adequate braid diagram with 4 strands representing $L$.
\end{corollary}


We speculate that it is possible to find new criteria for the crossing number of Khovanov adequate links, generalizing the result by Lickorish and Thistlethwaite \cite{LT}. We hope that it will stimulate research in this topic.

\subsection{Remarks on topological complexity}

Our initial approach to the complexity of finding the homotopy type of $I_D$ was via topological complexity of Smale \cite{Sm}. In that setting, an algorithm can be represented as a computation tree: a rooted tree growing down with the root (input) at the top and the leaves (output) at the bottom. Internal nodes are of two types: computational nodes and branching nodes, as shown in Figure \ref{nodes}. Computation nodes do not contribute to the topology of the computation tree, so the topological complexity of the tree is defined as the number of branching nodes. The topological complexity of a problem is the minimum topological complexity taken over all possible computation trees for that problem. 

\begin{figure}[h]
\centering
\includegraphics[width = 11cm]{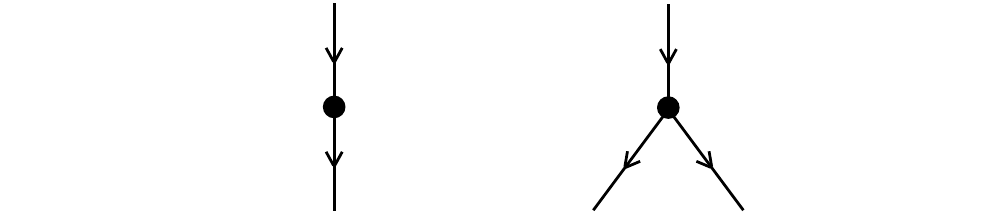}
\caption{\small{Nodes of a computational tree: computation node (left) and branching node (right).}}
\label{nodes}
\end{figure}

In our case, branching nodes are associated to wedge operations. In fact, the reason for attraction to topological complexity is visible through the whole paper: we often analyze a vertex $v$ of a graph $G$ and study the homotopy type of $I(G-v)$ and $I(G-st(v))$. In many case one of these simplicial complexes is contractible, so $I(G)\sim_h I(G-v)$ or $I(G)\sim_h \Sigma I(G-st(v))$. From the point of view of computation tree, this corresponds to a computation node, which is not increasing topological complexity (we have this situation for example if $v$ is a preleaf in $G$, as in Corollary \ref{leaf}). In some cases neither $I(G-v)$ nor $I(G-st(v))$ is contractible, but we can prove that $I(G-st(v))$ is contractible in $I(G-v)$, and therefore Proposition \ref{Prop3.3} leads to $I(G)\sim_h I(G-v) \vee \Sigma I(G-st(v))$, corresponding to a branching node in the computation tree. It was, at least initially, rather unexpected for us that branching nodes were used seldom. In fact, they were never used more than three times, and this is related to the fact that in Theorem \ref{theomain} the simplicial complex $I(w)$ has the homotopy type of the wedge of at most four spheres.

We decided not to include topological complexity in the paper because if we extend initial data in our problem (for example, including homotopy type of simple rhomboid graphs), then we can argue that topological complexity is equal to zero, so we cannot really say that in our case topological complexity is equal to three. On the other hand, if right after a branching node the complex in one of the branches is immediately contractible, still one can count the branching node in the computation tree and therefore its complexity will have linear grow.

\vspace{0.8cm}

\noindent \textbf{Acknowledgments:}
J. H. Przytycki is partially supported by the Simons Collaboration Grant 637794. M.Silvero is partially supported by Spanish Research Project PID2020-117971GB-C21, by IJC2019-040519-I, funded by MCIN/AEI/10.13039/501100011033 and by P20-01109 (JUNTA/FEDER).


\end{document}